\documentclass[]{imsart}
\usepackage{a4wide}
\usepackage[english]{babel}
\usepackage[latin1]{inputenc}
\usepackage{dsfont}
\usepackage{mathrsfs}
\usepackage{amsmath}
\usepackage{amssymb}
\usepackage{verbatim}
\usepackage{amsthm}
\usepackage{graphicx}
\usepackage{xcolor}
\usepackage{fancybox}
\usepackage[square,numbers]{natbib}
\usepackage{enumitem}
\usepackage{eurosym}

\usepackage{amsmath,amssymb,amsthm}
\usepackage{amscd}
\usepackage{color}
\usepackage{array, multirow}
\usepackage{mathcomp}

\usepackage{graphicx}
\usepackage{bbm}
\usepackage{enumitem}
\usepackage{enumerate}
\usepackage{mathtools}

\usepackage{pdfsync}%
\definecolor{darkblue}{rgb}{0.0,0.0,0.7}
\RequirePackage[%
colorlinks = true,%
linkcolor = darkblue,%
citecolor = darkblue,%
urlcolor = darkblue, %
]{hyperref}%
\hypersetup{%
  pdfauthor = {Bellec, Lecu\'e, Tsybakov},%
  pdfcreator = {pdflatex},%
  pdfproducer = {pdflatex},}

\usepackage{cleveref}

\usepackage{thmtools}

\numberwithin{equation}{section}

\newcommand{\norm}[1]{\|#1\|}%

\newcommand\eps{\varepsilon}

\DeclareMathOperator*{\argmin}{argmin}

\def\ds1{\textrm{1\kern-0.25emI}} 

\newcommand\E{{\mathbb E}}
\newcommand\R{{\mathbb R}}

\newcommand{\Pro}{\mathbb{P}}

\usepackage{marginnote}

\newcommand \cC{{\cal C}}

\newcommand \cN{{\cal N}}

\newcommand{\vf}{{\mathbf{f}}}

\newcommand{\vbeta}{{\boldsymbol{\beta}}}

\newcommand{\vlambda}{{\boldsymbol{\lambda}}}
\newcommand{\vv}{{\boldsymbol{v}}}

\newcommand{\ve}{{\boldsymbol{e}}}
\newcommand{\vzero}{{\boldsymbol{0}}}
\newcommand{\vx}{{\boldsymbol{x}}}
\newcommand{\vz}{{\boldsymbol{z}}}

\newcommand{\vu}{{\boldsymbol{u}}}
\newcommand{\vw}{{\boldsymbol{w}}}

\newcommand{\vy}{\mathbf{y}}
\newcommand{\vxi}{{\boldsymbol{\xi}}}
\newcommand{\vg}{{\boldsymbol{g}}}
\newcommand{\vdelta}{{\boldsymbol{\delta}}}
\newcommand{\vtau}{\boldsymbol{\tau}}
\newcommand{\vomega}{\boldsymbol{\omega}}

\newcommand{\design}{\mathbb{X}}

\newcommand{\hbeta}{{\boldsymbol{\hat\beta}}}

\newcommand{\norma}[1]{|#1|}

\declaretheorem[name=Theorem,numberwithin=section]{thm}
\declaretheorem[name=Theorem,sibling=thm]{theorem}
\declaretheorem[name=Proposition,sibling=thm]{proposition}
\declaretheorem[name=Lemma,sibling=thm]{lemma}
\declaretheorem[name=Corollary,sibling=thm]{corollary}

\declaretheorem[name=Remark,style=definition]{remark}

\newtheoremstyle{named}{}{}{\itshape}{}{\bfseries}{.}{.5em}{\thmnote{#3} #1}
\theoremstyle{named}
\newtheorem{condition}{condition}

\DeclareMathOperator{\Med}{Med}

\begin{document}

\begin{frontmatter}

\title{Slope meets Lasso: improved oracle bounds and optimality}
\runtitle{Slope meets Lasso}

\begin{aug}
\author{\fnms{Pierre C.} \snm{Bellec}${ }^{*,**,\dagger}$\ead[label=e1]{pierre.bellec@ensae.fr}},
\author{\fnms{Guillaume} \snm{Lecu{\'e}}${ }^{*,\dagger,\ddagger}$\ead[label=e2]{guillaume.lecue@ensae.fr}}
\and
\author{\fnms{Alexandre B.} \snm{Tsybakov}${ }^{*,\dagger}$
\ead[label=e3]{alexandre.tsybakov@ensae.fr}}

\runauthor{Bellec, Lecu\'e and Tsybakov}

\address{ENSAE$^*$, CREST (UMR CNRS 9194)$^\dagger$, CNRS$^{\ddagger}$ and Rutgers University$^{**}$}

\address{
\printead*{e1}\\
\printead*{e2}\\
\printead*{e3}}

\end{aug}

\runauthor{Bellec, Lecu\'e and Tsybakov}

\begin{abstract} \
We show that two polynomial time methods, a Lasso estimator with adaptively
    chosen tuning parameter and a Slope estimator, adaptively achieve the  minimax
    prediction and $\ell_2$ estimation rate $(s/n)\log (p/s)$ in
    high-dimensional linear regression on the class of $s$-sparse 
    vectors in $\R^p$. This is done under the Restricted Eigenvalue (RE) condition for the Lasso and under
    a slightly more constraining assumption on the design for the Slope.    The main results have the form of sharp oracle inequalities accounting for
    the model misspecification error.
    The minimax optimal bounds are also obtained for the $\ell_q$ estimation errors with $1\le q\le 2$
    when the model is well-specified.
    The results are non-asymptotic, and hold both in probability and in expectation.
    The assumptions that we impose on the design  are satisfied with high probability for a large class of random
    matrices with independent and possibly anisotropically distributed rows.
    We give a comparative analysis of conditions, under which oracle
    bounds for the Lasso and Slope estimators can be obtained.
    In particular,
    we show that several known conditions, such as the RE condition and the
    sparse eigenvalue condition are equivalent if the $\ell_2$-norms of regressors are
    uniformly bounded.
\end{abstract}



\end{frontmatter}

\section{Introduction}

One of the important issues in high-dimensional statistics is to construct 
methods that are both computable in polynomial time, and  have optimal
statistical performance in the sense that they attain the optimal convergence
rates on suitable classes of underlying objects (vectors, matrices), such as,
for example, the classes of $s$-sparse vectors.
It has been recently shown that, in some testing problems, this task cannot be
achieved, and there is a gap between the optimal rates in a minimax sense and
the best rate achievable by polynomial time algorithms \cite{MR3127849}.
However, the question about the existence of such a gap remains open for the
most famous problem, namely, that of estimation and prediction in high-dimensional linear regression on
the classes of $s$-sparse parameters in $\R^p$.
The known polynomial time methods such as the Lasso, the Dantzig selector and
several other were shown to attain the prediction or $\ell_2$-estimation rate
$(s/n)\log(p)$ \cite{MR2300700,MR2533469} while the minimax rate for the
problem is $(s/n)\log(p/s)$ (cf. \cite{rigollet2011exponential,lptv2011,ye_zhang2010,raskutti2011,abramovich2010,verzelen2012,candes_davenport2013}
and Section~\ref{sec:lower} below).
The recent papers \cite{MR3485953,LM_reg_sparse} inspire hope that computationally
feasible methods can achieve the minimax rate $(s/n)\log (p/s)$.
Specifically, \cite{MR3485953} shows that for a particular random design (i.i.d.
standard normal regressors) the rate $(s/n)\log (p/s)$ is asymptotically
achieved by a Slope estimator, which is computable in polynomial time. 
An extension of \cite{MR3485953} to subgaussian designs is given in  \cite{LM_reg_sparse} that provides a non-asymptotic bound with the same rate. However, akin to \cite{MR3485953}, a key assumption in  \cite{LM_reg_sparse} is that the design is isotropic, so that its covariance matrix is proportional to the identity matrix.

The Slope estimator suggested in \cite{MR3418717} is defined as a solution of the convex
minimization problem given in \eqref{eq:def-hbeta} below.
This estimator requires $p$ tuning parameters $\lambda_1,\dots,\lambda_p$
not all equal to 0 and such that $\lambda_1\ge\dots\ge \lambda_p\ge0$.
These $\lambda_1,\dots,\lambda_p$ are the weights of the sorted $\ell_1$ norm, cf.
\eqref{eq:def-norm_*} below.

In this paper, we show that under a Restricted Eigenvalue ($RE$) type condition on the design,
the Slope estimator with suitably chosen tuning parameters achieves the optimal rate $(s/n)\log(p/s)$ for both
the prediction and the $\ell_2$ estimation risks, and both in probability and in expectation. The recommended tuning parameters are given in \eqref{eq:recommendation-weights-slope} below.
Furthermore, we show that a large class of random design matrices with independent and possibly anisotropically distributed rows
satisfies this $RE$ type condition with high probability. In other words, our conditions on the design for the Slope estimator are very close to those usually assumed for the Lasso estimator while the rate  improves from $(s/n)\log(p)$ (previously known for the Lasso) to the optimal rate $(s/n)\log(p/s)$.
Next, with the same method of proof, we show that the Lasso estimator also achieves this improved (and optimal) rate when the sparsity $s$ is known.  If $s$ is unknown, we propose to replace $s$ by an estimator $\hat s$ such that the
        bound $\hat s \le s$ holds with high probability.
        We show that the suggested $\hat s$ is such that the Lasso estimator
        with tuning parameter of order $\sqrt{\log(p/\hat s) /n }$
        achieves the optimal rate $(s/n)\log(p/s)$.

The main results are obtained in the form of sharp oracle inequalities accounting for the
model misspecification error.
The minimax optimal bounds are also established for the $\ell_q$-estimation errors with $1\le q\le 2$ when the model is well-specified.
All our results are non-asymptotic.

As a by-product, we cover some other related issues of independent interest:
\begin{itemize}
    \item
We give a comparative analysis of conditions, under which oracle
        bounds for the Lasso and Slope estimators can be obtained showing, in
        particular, that several known conditions are equivalent.
        \item Due to the new techniques, we obtain bounds in probability with fast rate $(s/n)\log(p/s)$ at any level of confidence while using the same tuning parameter. As opposed to the previous work on the Lasso, the level of confidence is not linked to the tuning parameter of the method. As a corollary, this implies rate optimal bounds on any moments of the estimation and prediction errors.
\end{itemize}

\section{Statement of the problem and organization of the paper}
\label{sec:statement2}

Assume that we observe the vector
\begin{equation*}
    \vy = \vf + \vxi,
\end{equation*}
where $\vf\in\R^n$ is an unknown deterministic mean
and $\vxi\in\R^n$ is a noise vector. Let $\sigma>0$.
Everywhere except for Section \ref{sec:extension_to_a_sub_gaussian_noise} we assume that $\vxi$ is normal $\cN(\vzero,\sigma^2 I_{n\times n})$, where  $I_{n\times n}$ denotes the $n\times n$ identity matrix.

For all $\vu = (u_1,\dots,u_n)\in\R^n$, define the empirical norm of $\vu$
by
\begin{equation*}
    \norm{\vu}_n^2 = \frac 1 n \sum_{i=1}^n u_i^2.
\end{equation*}

Let $\design\in\R^{n\times p}$ be a given matrix that we will call the design matrix. If $\hbeta=\hbeta(\vy)$ is an estimator valued in $\R^p$,
the value $\design\hbeta$ is used as a prediction for $\vf$.
The prediction error of an estimator $\hbeta$ is given by
$\norm{\design\hbeta - \vf}_n^2$.
If the model is well-specified, that is $\vf=\design\vbeta^*$ for some $\vbeta^*\in\R^p$, then
$\hbeta$ is used as an estimator of $\vbeta^*$.
The estimation error of $\hbeta$ is given by
$\norma{\hbeta - \vbeta^*}_q^q$ for some $q\in[1,2]$, where $\norma{\cdot}_q$ denotes the $\ell_q$-norm in $\R^p$.

Two estimators will be studied in this paper: the Lasso estimator and the Slope estimator.
The Lasso estimator $\hbeta$ is
a solution of the minimization problem
\begin{equation}\label{eq:def-lasso}
    \hbeta \in \argmin_{\vbeta\in\R^p}
    \left(
        \norm{\design\vbeta - \vy}_n^2
        + 2 \lambda\norma{\vbeta}_1
    \right),
\end{equation}
where $\lambda> 0$ is a tuning parameter.
Section \ref{sec:lasso-4} studies the prediction and estimation performance
of the Lasso estimator with tuning parameter
of order $\sigma\sqrt{\log(p/s)/n}$, where $s\in\{1,\dots,p\}$
is a sparsity parameter which is supposed to be known.
In Section \ref{sec:lepski} we propose an adaptive choice of this parameter.
Section \ref{sec:lepski} defines an estimator $\hat s$ valued in $\{1,\dots,p\}$
and studies the performance of the Lasso estimator
with a data-driven tuning parameter of order $\sigma\sqrt{\log(p/\hat s)/n}$.

Section \ref{sec:slope} studies the Slope estimator \cite{MR3418717},
which is defined as follows.
Let $\vlambda = (\lambda_1,\dots,\lambda_p) \in \R^p$
be a vector of tuning parameters not all equal to 0 such that
$\lambda_1\ge \lambda_2\ge\dots\ge \lambda_p\ge0$.
For any $\vbeta=(\beta_1,\dots,\beta_p)\in\R^p$, let $(\beta_1^\sharp,\dots,\beta_p^\sharp)$ be
a non-increasing rearrangement of
$
    |\beta_1|,\dots,|\beta_p|
$.
Set
\begin{equation}
    \norma{\vbeta}_* = \sum_{j=1}^p \lambda_j \beta_j^\sharp,
    \qquad
    \vbeta\in\R^p,
    \label{eq:def-norm_*}
\end{equation}
which defines a norm on $\R^p$, cf.
\cite[Proposition 1.2]{MR3418717}.
Equivalently, we can write
\begin{equation}
    \norma{\vbeta}_*
    =
    \max_{\phi}
    \sum_{j=1}^p \lambda_j |\beta_{\phi(j)}| ,
    \label{eq:maximal-rearrange}
\end{equation}
where the maximum is taken over all permutations $\phi=(\phi(1),\dots,\phi(p))$ of $\{1,\dots,p\}$.
The Slope estimator $\hbeta$ is defined as a solution of
the minimization problem
\begin{equation}\label{eq:def-hbeta}
    \hbeta \in \argmin_{\vbeta\in\R^p}
    \left(
        \norm{\design\vbeta - \vy}_n^2
        + 2 \norma{\vbeta}_*
    \right)
    .
\end{equation}
Section \ref{sec:slope} establishes oracle inequalities
and estimation error bounds for the Slope estimator with tuning parameters
\begin{equation}
    \lambda_j = A \sigma \sqrt{\frac{\log(2p/j)}{n}},
    \qquad
    j=1,\dots,p,
    \label{eq:recommendation-weights-slope}
\end{equation}
{
for any constant $A> 4 + \sqrt 2$.
}

Section \ref{sec:lower} gives non-asymptotic minimax lower bounds
showing that the upper bounds of Sections \ref{sec:lasso-4} -- \ref{sec:slope} cannot be improved.
In Section \ref{sec:assumptions}, we provide a comparison of the conditions on the design matrix $\design$, under which the results are obtained. In particular, we prove that
the oracle inequalities
for the Slope estimator in Section \ref{sec:slope}
hold for design matrices with independent and possibly anisotropically distributed subgaussian rows.  Section \ref{sec:extension_to_a_sub_gaussian_noise} explains that, up to changes in numerical constants, all results
of the paper remain valid if the components of the noise vector $\vxi$
are independent subgaussian random variables.
The proofs are given in the Appendix.

\section*{Notation and preliminaries}

We will assume that the diagonal elements of the Gram matrix $\frac 1 n \design^T\design$ are at most 1,
that is,
$\max_{j=1,\dots,p}\norm{\design\ve_j}_n \le 1$
where $(\ve_1,\dots,\ve_p)$ is the canonical basis in $\R^p$.
Let
$\vg = (g_1,\dots,g_p)$ be the random vector with components
\begin{equation}
    g_j = \frac 1 {\sqrt n} \vx_j^T\vxi
    \qquad
    \text{where}
    \qquad
    \vx_j
    = \design \ve_j, \quad j=1,\dots,p.
    \label{eq:def-g_j}
\end{equation}
If $\vxi\sim \cN(\vzero,\sigma^2 I_{n\times n})$ it follows from the inequality $\norm{\vx_j}_n\le 1$ that the random variables $g_j$ are zero mean Gaussian with variance at most $\sigma^2$.
We denote by $\vg^\sharp = (g^\sharp_1,\dots,g^\sharp_p)$ a non-increasing rearrangement of $(|g_1|,\dots,|g_p|)$.
We also use the notation
\begin{equation*}
    \norma{\vbeta}_0=\sum_{j=1}^p I(\beta_j \ne 0),
    \qquad \norma{\vbeta}_\infty=\max_{j=1, \dots, p} |\beta_j|, \qquad \text{and}
    \qquad \norma{\vbeta}_q=\left(\sum_{j=1}^p|\beta_j|^q\right)^{1/q}
\end{equation*}
for $0<q<\infty$. Here, $I(\cdot)$ is the indicator function.
For any set $J\subset\{1,\dots,p\}$, denote by $J^c$ its complement, by $|J|$ its cardinality, and for any $\vu=(u_1,\dots,u_p)\in\R^p$,
let $\vu_J\in\R^p$ be the vector such that its $j$th component is equal to $u_j$ if
$j\in J$ and equal to 0 otherwise.
For two real numbers $a,b$, we will use the notation
$
    a\vee b = \max(a,b).
$

We denote by $\Med[Z]$ a median  of a real valued random variable $Z$, that  is, any real number such that
$\mathbb P(Z \ge \Med[Z]) \ge 1/2$
and
$\mathbb P(Z \le \Med[Z]) \ge 1/2$.

The following bounds on
the sum $\sum_{j=1}^s\log(2p/j)$ will be useful.
From Stirling's formula, we easily deduce that
$s\log (s/e) \le \log( s! ) \le s \log(s)$
and thus
\begin{equation}\label{eq:algebra_stirling}
    s \log(2p/s) \le
    \sum_{j=1}^s
    \log(2p/j)
    =
    s \log(2p)
    - \log(s!)
    \le
    s \log(2ep/s).
\end{equation}
Finally, for a given $\delta_0\in(0,1)$ and for any $\vu=(u_1,\dots,u_p) \in\R^p$ we set
\begin{equation}
    H(\vu)
    \triangleq
    {(4+\sqrt 2)} \sum_{j=1}^p u_j^\sharp \sigma \sqrt{\frac{ \log(2p/j)}{n}},
    \qquad
    G(\vu)
    \triangleq
    {(4+\sqrt 2)} \sigma \sqrt{\frac{\log(1/\delta_0)}{n}}
    \norm{\design\vu}_n,
    \label{eq:def-H-G}
\end{equation}
where $(u_1^\sharp,\dots,u_p^\sharp)$ is a non-increasing
rearrangement of $(|u_1|,\dots,|u_p|)$.

\section{The tuning parameter of the Lasso need not be tied to a confidence level}

In this section, we denote by $\hbeta$ the Lasso estimator
defined by \eqref{eq:def-lasso} and
provide
improved probability estimate for the performance of the Lasso estimator with tuning parameter of order $\sigma\sqrt{2\log p}$.
First, we state a version of the Restricted Eigenvalue condition that we will refer to in the sequel. Let $s\in\{1,\dots,p\}$, and let
$c_0>0$ be a constant.

\begin{condition}[$RE(s,c_0)$]
    \label{assumption:RE-classic}
    The design matrix $\design$ satisfies $\norm{\design\ve_j}_n \le 1$ for all $j=1,\dots,p$,
    and
    \begin{equation}\label{eq:RE-classic}
        \kappa(s,c_0) \triangleq \inf_{\vdelta\in \cC_{RE}(s,c_0): \vdelta\ne \vzero   }
         \frac{\norm{\design\vdelta}_n}{\norma{\vdelta}_2}
        > 0,
    \end{equation}
    where $\cC_{RE}(s,c_0) = \{
            \vdelta\in\R^p:
            \norma{\vdelta}_1 \leq (1+c_0) \sum_{j=1}^s \delta_j^\sharp
    \}$
    and $\delta_1^\sharp\geq \cdots\geq \delta_p^\sharp$ denotes
    a non-increasing rearrangement of $|\delta_1|, \ldots, |\delta_p|$.
\end{condition}
Though stated in somewhat different form, inequality \eqref{eq:RE-classic} is equivalent to the original $RE$ condition of \cite{MR2533469}. Indeed, let $\vdelta\in\R^p$, and let  $J_{*}=J_{*}(\vdelta){\subseteq} \{1,\dots,p\}$ be the set of indices of the $s$ largest in absolute value components of $\vdelta$.  Then
$\sum_{j=1}^s\delta_j^\sharp = \norma{\vdelta_{J_{*}}}_1$.
Therefore, the condition
$
\norma{\vdelta}_1 \leq (1+c_0) \sum_{j=1}^s \delta_j^\sharp
$
can be written as
$\norma{\vdelta_{J_*^c}}_1 \le c_0 \norma{\vdelta_{J_*}}_1$.  {Thus, an equivalent form of \eqref{eq:RE-classic} is obtained by replacing the cone
$\cC_{RE}(s,c_0)$ with}
\begin{equation}\label{C}
   { \cC_{RE}'}(s,c_0)
    = \cup_{J{\subseteq}\{1,\dots,p\}: |J|\le s} \{
            \vdelta\in\R^p:
            \norma{\vdelta_{J^c}}_1
            \le
            c_0
            \norma{\vdelta_{J}}_1
    \},
\end{equation}
{which is the standard cone of the $RE$ condition as introduced in \cite{MR2533469}}. One minor difference from \cite{MR2533469} is that in \eqref{eq:RE-classic}
we have $|\vdelta|_2$ rather than $|\vdelta_{J_*}|_2$ in the denominator. This only modifies the constant $\kappa(s, c_0)$ by factor $\sqrt{1+c_0}$. Indeed, for any $\vdelta \in\cC_{RE}'(s,c_0)$,
\begin{equation*}
 \norma{\vdelta_{J_*^c}}_2^2 = \sum_{j\in J_*^c}\delta_j^2 \leq \sum_{j\in J_*^c}|\delta_j| \Big(\frac{1}{|J_*|}\sum_{k\in J_*}|\delta_k|\Big)\leq \frac{\norma{\vdelta_{J_*}}_1 \norma{\vdelta_{J_*^c}}_1}{|J_*|}\leq \frac{c_0 \norma{\vdelta_{J_*}}_1^2}{|J_*|}\leq c_0 \norma{\vdelta_{J_*}}_2^2,
 \end{equation*}
 and thus $\norma{\vdelta}_2 = (\norma{\vdelta_{J_*}}_2^2+\norma{\vdelta_{J_*^c}}_2^2)^{1/2}\leq \sqrt{1+c_0}\norma{\vdelta_{J_*}}_2$.
Another difference from \cite{MR2533469} is that we include the assumption $\norm{\design\ve_j}_n \le 1$
in the statement of the $RE$ condition. This is a mild assumption that is omnipresent in the literature on the Lasso.
Often it is stated with equality and is interpreted as a normalization.

Consider the following result based on \cite{koltchinskii2011nuclear,MR2533469,dalalyan2017prediction},
which is representative of the non-asymptotic bounds obtained so far for the prediction performance of the Lasso.

\begin{proposition}[\cite{koltchinskii2011nuclear,MR2533469,dalalyan2017prediction}]
    \label{prop:previous}
    Let $p\ge4$,
    $s\in\{1,...,p\}$, $\eps>0$ and set $c_0=1+1/\epsilon$.
    For any $\delta\in(0,1/2)$ and $\eps\in(0,1)$, the Lasso estimator \eqref{eq:def-lasso} with tuning parameter
    \begin{equation}
        \label{old-tuning}
        \lambda = (1+\eps)\sigma \sqrt{2\log(p/\delta)/n}
    \end{equation}
    satisfies with probability $1-2\delta$ the oracle inequality
    \begin{equation}
        \norm{\design\hbeta - \vf}_n^2
        \le
        \inf_{\vbeta\in\R^p: \norma{\vbeta}_0\le s}
        \norm{\design\vbeta - \vf}_n^2
        + \frac{\sigma^2}{n} \left(\frac{(1+\eps)\sqrt{2s\log(p/\delta)}}{\kappa(s,c_0)} + \sqrt s + \sqrt{2\log(1/\delta)} \right)^2.
        \label{oi-previous}
    \end{equation}
\end{proposition}
An oracle inequality of the same kind as \eqref{oi-previous} was first obtained  in \cite[Theorem 6.1]{koltchinskii2011nuclear} and in a slightly less general form,
with some factor $C>1$ in front of $\norm{\design\vbeta - \vf}_n^2$ in  \cite{MR2533469}.
The numerical constants in Proposition \ref{prop:previous} are taken from the proof of Theorem 3 in \cite{dalalyan2017prediction} (cf. the inequality preceding (25) in \cite{dalalyan2017prediction}).

A notable feature of \Cref{prop:previous} and of other non-asymptotic bounds for the Lasso available in the literature is that the confidence level $1-2\delta$
is tied to the tuning parameter $\lambda$,  cf. \eqref{old-tuning}.
If a confidence level closer to one is desired (i.e., smaller $\delta$),
the previous results suggest that the tuning parameter should be increased according to the relationship \eqref{old-tuning}
between $\lambda$ and $\delta$. We claim that it is not needed. Indeed, the following proposition holds.

\begin{proposition}
    \label{prop:previous-improved}
    Let $p\ge 2$, $s\in\{1,...,p\}$, $\eps > 0$ and set $c_0=1+1/\epsilon$.
    Let $\hbeta$ be the Lasso estimator \eqref{eq:def-lasso} with tuning parameter
    \begin{equation}
        \lambda = (1+\eps)\sigma \sqrt{2\log(p)/n}.
        \label{new-tuning}
    \end{equation}
    Then for any $\delta\in(0,1)$,
    \begin{equation}
        \norm{\design\hbeta - \vf}_n^2
        \le
        \inf_{\vbeta\in\R^p: \norma{\vbeta}_0\le s}
        \norm{\design\vbeta - \vf}_n^2
        + \frac{\sigma^2}{n} \left(\frac{(1+\eps)\sqrt{2s\log p}}{\kappa(s,c_0)} + \sqrt s + \sqrt{2\log(1/\delta)}   + 2.8\right)^2.
        \label{eq:oi-previous-improved}
    \end{equation}
    holds with probability at least $1-\delta$.
\end{proposition}
The proof of \Cref{prop:previous-improved} is given in \Cref{appendix-previous-improved}.
Let us highlight some features of \Cref{prop:previous-improved} that are new.

First, the tuning parameter \eqref{new-tuning} does not depend on the confidence constant $\delta$.
The Lasso estimator with tuning parameter \eqref{new-tuning} enjoys the oracle inequality \eqref{eq:oi-previous-improved} for any
$\delta\in (0,1)$. This contrasts with \Cref{prop:previous} where the oracle inequality is for a single $\delta$ tied to the tuning parameter \eqref{old-tuning}. As a consequence, \eqref{eq:oi-previous-improved} immediately implies bounds for all moments of $ \norm{\design\hbeta - \vf}_n^2$, which cannot be obtained from Proposition \ref{prop:previous}.

Second, the probability that an oracle inequality
with error term of order $s\log(p)/n$ holds
is substantially closer to 1 than it was commonly understood before.
For example, take $\delta = p^{-s}$, which balances the remainder term in \eqref{eq:oi-previous-improved}.  Then,  \Cref{prop:previous-improved} yields that the Lasso estimator
with tuning parameter \eqref{new-tuning} converges with the rate smaller than $s\log(p)/n$ up to a multiplicative constant.
On the other hand, the choice $\delta= p^{-s}$ in \Cref{prop:previous}
yields only a suboptimal rate of order $s^2\log(p)/n$.

The constant term 2.8 in \eqref{eq:oi-previous-improved} is negligible.
Thus, in asymptotic regimes where $p\rightarrow \infty$ or $\delta\rightarrow 0$, the constant term 2.8 is
dominated by $\sqrt{\log p}$ or $\sqrt{\log(1/\delta)}$. If we ignore this constant term, the bound \eqref{eq:oi-previous-improved} strictly improves upon \eqref{oi-previous}.


In spite of these improvements, the result of  \Cref{prop:previous-improved} is not completely satisfying since only the rate $s\log(p)/n$ and not the optimal rate $(s/n)\log(p/s)$ is proved. In the next section,
we show that the optimal rate can be achieved by the Lasso estimator with tuning parameter of the order $\sigma\sqrt{\log(p/s)/n}$.

\section{Optimal rates  for the Lasso estimator}
\label{sec:lasso-4}

In this section, we denote by $\hbeta$ the Lasso estimator
defined by \eqref{eq:def-lasso}, and we derive upper bounds for its prediction and estimation errors.
As usual in the Lasso context, the argument contains two main ingredients. First, all randomness is removed from the problem by reducing the consideration to a suitably chosen random event of high probability. Second, the error bounds are derived on this event by a purely deterministic argument. In our case, such a deterministic argument is given in Theorem~\ref{thm:lasso-main-event} below, while the ``randomness removing tool'' is provided by the next theorem. As we will see in Section \ref{sec:slope}, this theorem is
common to the study of both the Lasso and the Slope estimators.

\begin{thm}
    \label{thm:main-event}
    Let $\delta_0\in(0,1)$ and let $\design\in \R^{n\times p}$ be a matrix such that $\max_{j=1,\dots,p}\norm{\design\ve_j}_n \le 1$.
    Let $H(\cdot)$ and $G(\cdot)$ be defined in \eqref{eq:def-H-G}.
    If $\vxi\sim\mathcal N(\vzero,\sigma^2 I_{n\times n})$,
     then the random event
    \begin{equation}
        \left\{
                \frac 1 n \vxi^T\design \vu
                \le
                \max\Big(
                    H(\vu)
                    ,
                    G(\vu)
                \Big), \ \forall \vu\in\R^p
        \right\}
        \label{eq:main-event}
    \end{equation}
    is of probability at least $1-{\delta_0/2}$.
\end{thm}
The proof of Theorem \ref{thm:main-event} is given in Appendix \ref{s:proof-main-event}.
We now discuss consequences of Theorem \ref{thm:main-event} leading to tighter bounds. We will need the following condition on the design matrix $\design$ that will be called the {\it Strong Restricted Eigenvalue} condition or shortly the $SRE$ condition. Let $c_0>0$ and $s\in\{1,\dots,p\}$ be fixed.
\begin{condition}[$SRE(s,c_0)$]
    \label{assumption:strong-re-lasso}
    The design matrix $\design$ satisfies $\norm{\design\ve_j}_n \le 1$ for all $j=1,\dots,p$,
    and
    \begin{equation}\label{def:SRE}
        \theta(s,c_0) \triangleq \min_{\vdelta\in\cC_{SRE}(s,c_0):\vdelta\ne\vzero} \frac{\norm{\design\vdelta}_n}{\norma{\vdelta}_2}
        > 0,
    \end{equation}
    where
    $
    \cC_{SRE} (s,c_0)  \triangleq\left\{
                \vdelta\in\R^p: \norma{\vdelta}_1 \le (1+c_0) \sqrt s \norma{\vdelta}_2
        \right\}
        $
        is a cone in $\R^p$.
\end{condition}

Inequality \eqref{def:SRE} differs from its analog \eqref{eq:RE-classic} in the $RE(s,c_0)$ condition only in the definition of the cone $\cC_{SRE} (s,c_0)$, {and in general} \eqref{def:SRE} is more constraining.
Indeed, the cone $\cC_{RE}(s,c_0)$ of the $RE(s,c_0)$ condition is the set of all $\vdelta\in\R^p$ such that
$    \norma{\vdelta}_1 \le (1+c_0) \sum_{j=1}^s \delta_j^\sharp.
$
By the Cauchy-Schwarz inequality,
$ \sum_{j=1}^s \delta_j^\sharp \le  \sqrt s (\sum_{j=1}^s (\delta_j^\sharp)^2)^{1/2} \le \sqrt s \norma{\vdelta}_2$
so that
\begin{equation}\label{CC}
    \cC_{RE}(s,c_0)
    \subseteq
    \cC_{SRE}(s,c_0).
\end{equation}

Note that we have included the requirement that $\norm{\design\ve_j}_n \le 1$
for all $j$ in the $RE$ and the $SRE$ conditions.
It can be replaced by $\norm{\design\ve_j}_n \le \theta_0$ for some
$\theta_0>0$ but for brevity and w.l.o.g. we take here $\theta_0=1$.
Interestingly, due to the inclusion of the assumption $\norm{\design\ve_j}_n \le 1$,
the $RE$ condition becomes equivalent to the $SRE$ condition up to absolute constants.
Moreover, the equivalence further extends to the $s$-sparse eigenvalue
condition.
This is detailed in  Section \ref{sec:assumptions} below.

Under the $SRE$ condition, we now establish a deterministic result, which is central in our argument.
We first introduce some notation.
Let $\gamma\in(0,1)$ be a constant.
For any tuning parameter $\lambda>0$, set
\begin{equation}
    \delta(\lambda)
    \triangleq
    \exp\left( - \left(\frac{\gamma \lambda\sqrt{n}}{{(4+\sqrt 2)}\sigma}\right)^2 \right)
    \qquad
    \text{ so that }
    \qquad
    \lambda = \frac{{(4+\sqrt 2)}\sigma}{\gamma} \sqrt{\frac{\log({1}/\delta(\lambda))}{n} }.
\end{equation}
For given $s\in\{1,\dots,p\}$,
the following theorem holds under the condition
\begin{equation}
    \lambda \ge \frac{{(4+\sqrt 2)}\sigma}{\gamma} \sqrt{\frac{\log(2ep/s)}{n}}
    \qquad
    \text{or equivalently}
    \qquad
    \delta(\lambda) \le {s/(2ep)}.
    \label{eq:tuning-lambda}
\end{equation}

\begin{thm}\label{thm:lasso-main-event}
    Let $s\in\{1,\dots,p\}$, $\gamma\in (0,1)$ and $\tau\in[0,1 - \gamma)$.
    Assume that the $SRE(s,c_0)$ condition holds with
    $c_0 = c_0(\gamma,\tau) = \frac{ 1+ \gamma  + \tau}{1 - \gamma - \tau}$.
    Let $\lambda$ be a tuning parameter such that
    \eqref{eq:tuning-lambda} holds.
    Let $\delta_0\in(0,1)$.
    Then, on the event \eqref{eq:main-event},
    the Lasso estimator $\hbeta$
    with tuning parameter $\lambda$
    satisfies
    \begin{equation}
        2\tau \lambda \norma{\hbeta - \vbeta}_1
        +
        \norm{\design\hbeta - \vf}^2_n
        \le
         \norm{\design\vbeta - \vf}^2_n
         + C_{\gamma,\tau}(s,\lambda,\delta_0)
          \lambda^2 s          ,
        \label{eq:soi-lasso1}
    \end{equation}
    for all $\vbeta\in\R^p$ such that $\norma{\vbeta}_0\le s$, and all $\vf\in\R^n$,
    where
    \begin{equation*}
        C_{\gamma,\tau}(s,\lambda,\delta_0)
        \triangleq
        (1+ \gamma + \tau)^2
        \left(\frac{  \log({1}/\delta_0)}{s\log({1}/\delta(\lambda))} \vee \frac1{\theta^2(s,c_0(\gamma,\tau))  } \right).
    \end{equation*}
    Furthermore, if $\vf = \design\vbeta^*$ for some $\vbeta^*\in\R^p$ with $\norma{\vbeta^*}_0 \le s$
    then on the event \eqref{eq:main-event}, we have for any $1\leq q\leq2$,
    \begin{equation}\label{eq:lasso_ell_q}
    \norma{\hbeta-\vbeta^*}_q\leq \Big(\frac{C_{\gamma, \tau}}{2\tau}\Big)^{2/q-1}\Big(\frac{C_{\gamma, 0}}{1+\gamma}\Big)^{2-2/q}\lambda s^{1/q}.
    \end{equation}
\end{thm}

Theorem \ref{thm:lasso-main-event} is proved in Appendix \ref{sec:proof-lasso}.
Before the statement of its corollaries, a few comments are in order.

The conclusions of Theorem \ref{thm:lasso-main-event} hold on the event \eqref{eq:main-event}, which is independent of $\gamma$ and $\tau$.
Thus, on the event \eqref{eq:main-event}, for all choices of $\tau,\gamma$ and $\lambda$ such that \eqref{eq:tuning-lambda} holds,
the oracle inequality \eqref{eq:soi-lasso1}
and the estimation bound \eqref{eq:lasso_ell_q} are satisfied.

The constants $\gamma$ and $\tau$ are
such that $\gamma+\tau<1$.
For the ease of {presentation}, the particular choice $\gamma = 1/2$ and $\tau = 1/4$ will be used below
to derive two corollaries of Theorem \ref{thm:lasso-main-event}.
If $\gamma=1/2$ and $\tau = 1/4$, then the constants in Theorem  \ref{thm:lasso-main-event}
have the form
\begin{equation}
    c_0 = 7,
    \qquad
    (1+\gamma + \tau)^2 = \frac{49}{16},
    \qquad
    \frac{1+\gamma}{1-\gamma} = 3,
    \qquad
    1 +  \gamma = 3/2,
    \label{eq:simple-ineq-constants}
\end{equation}
while inequality \eqref{eq:lasso_ell_q} can be transformed into
\begin{equation}\label{eq:lasso_ell_q-bis}
    \norma{\hbeta-\vbeta^*}_q\leq  \frac{49}{8}\left(\frac{  \log({1}/\delta_0)}{s\log({1}/\delta(\lambda))} \vee \frac1{\theta^2(s,7)  } \right)\lambda s^{1/q},
    \end{equation}
where we have used that $\theta^2(s,c_0)\le \theta^2(s,\frac{ 1+ \gamma }{1 - \gamma })$.

We now take a closer look at the constant $C_{\gamma,\tau}(s,\lambda,\delta_0)$.
This constant is always greater than or equal to $(1+\gamma+\tau)^2/\theta^2(s,c_0)$.
Furthermore, the value
\begin{equation}
    \delta_0^* = \left(\delta(\lambda)\right)^{\frac{s}{\theta^2(s,c_0)}}
    \label{eq:choice-explicit-delta0}
\end{equation}
is the smallest $\delta_0\in(0,1)$ such that
$C_{\gamma,\tau}(s,\lambda,\delta_0) = (1+\gamma+\tau)^2/\theta^2(s,c_0)$.
If $\lambda$ satisfies \eqref{eq:tuning-lambda},
then
$$
    \delta_0^* \le \Big(\frac{s}{2ep}\Big)^{\frac{s}{\theta^2(s,c_0)}}.
$$
Using these remarks we obtain the following corollary
of Theorems \ref{thm:main-event} and \ref{thm:lasso-main-event}
with the choice $\gamma=1/2,\tau=1/4$, and $\delta_0 = \delta_0^*$.

\begin{corollary}
    \label{cor:lasso-deviation}
    Let $s\in\{1,\dots,p\}$.
    Assume that the
    $SRE(s,7)$ condition holds.
    Let $\hbeta$ be the Lasso estimator with tuning parameter $\lambda$ satisfying
    \eqref{eq:tuning-lambda} for $\gamma = 1/2$.
    Then, with probability at least $1-{\frac 1 2} \left(\frac{s}{2ep}\right)^{\frac{s}{\theta^2(s,7)}}$,
    we have
    \begin{align}
        \frac \lambda 2 \norma{\hbeta - \vbeta}_1
        +
        \norm{\design\hbeta - \vf}^2_n
        \le
         \norm{\design\vbeta - \vf}^2_n
         + \frac{49 \lambda^2 s}{16 \theta^2(s,7)}
                 \label{eq:loasso-explicit-constants-soi}
    \end{align}
    for all $\vbeta\in\R^p$ such that $\norma{\vbeta}_0\le s$, and all $\vf\in\R^n$.
    Furthermore, if $\vf = \design\vbeta^*$ for some $\vbeta^*\in\R^p$ with $\norma{\vbeta^*}_0 \le s$
    then, for any $1\leq q\leq2$,
    \begin{align}
        \mathbb P\left(
        \norma{\hbeta-\vbeta^*}_q
        \le \frac{49 \lambda s^{1/q}}{8\theta^2(s,7)}
         \right)
         \ge 1-{\frac 1 2} \left(\frac{s}{2ep}\right)^{\frac{s}{\theta^2(s,7)}}.
        \label{eq:loasso-explicit-constants-estimation-ell2}
    \end{align}
\end{corollary}

Since $\theta^2(s,7)\le 1$, the probability in Corollary~\ref{cor:lasso-deviation} is greater than $1-{\frac 1 2 }\left(\frac{s}{2ep}\right)^s.$
If the tuning parameter is chosen
such that \eqref{eq:tuning-lambda} holds with equality,
then
$\lambda^2s$ is equal to
\begin{equation*}
    \frac{\sigma^2s\log(2ep/s)}{n}
\end{equation*}
up to a multiplicative constant.
This is the minimax rate with respect to the prediction error over
the class of all $s$-sparse vectors $B_0(s) = \{\vbeta\in\R^p: \norma{\vbeta}_0\le s\}$. The rate $ \lambda s^{1/q}$ in \eqref{eq:loasso-explicit-constants-estimation-ell2} is minimax optimal for the $\ell_q$ estimation problem.
A more detailed discussion of the minimax rates
is given in \Cref{sec:lower}.

Finally, the conclusions of Theorem \ref{thm:lasso-main-event}
hold for all $\delta_0 \le \delta_0^*$.
This allows us to integrate the oracle inequality \eqref{eq:soi-lasso1}
and the estimation bound \eqref{eq:lasso_ell_q} to obtain
the following results in expectation.

\begin{corollary}
    \label{cor:lasso-E}
    Let $s\in\{1,\dots,p\}$.
    Assume that the
    $SRE(s,7)$ condition holds.
    Let $\hbeta$ be the Lasso estimator with tuning parameter $\lambda$ satisfying
    \eqref{eq:tuning-lambda} for $\gamma = 1/2$.
    Then,
    \begin{equation}
        \E \left[ \frac \lambda 2 \norma{\hbeta - \vbeta}_1
                +
                \norm{\design\hbeta - \vf}^2_n
        \right]
        \le
        \norm{\design\vbeta - \vf}^2_n
        + \frac{49\lambda^2 s}{16} \left(
            \frac{1}{\theta^2(s,7)}
            + \frac{1}{{2}\log(2ep)}
        \right)
        \label{eq:soi-expectation-lasso}
    \end{equation}
    for all $\vbeta\in\R^p$ such that $\norma{\vbeta}_0\le s$, and all $\vf\in\R^n$.
    Furthermore, if $\vf=\design\vbeta^*$
    for some $\vbeta^*\in\R^p$ with $\norma{\vbeta^*}_0 \le s$
    then, for any $1\leq q\leq2$,
    \begin{equation}
        \E \left[ \norma{\vbeta^*-\hbeta}_q^q \right]
        \le
        \frac{(49)^q\lambda^q s}{8^q}
        \left(
            \frac{1}{\theta^{2q}(s,7)}
            +
            \frac{{1}}{(\log(2ep))^q}
        \right)
        .
        \label{eq:esimtation-ell2-expectation-lasso}
    \end{equation}
\end{corollary}
Corollary \ref{cor:lasso-E} is proved in Section \ref{sec:proof-lasso}.

\vspace{-2mm}

\begin{remark}
    The smallest values of $\lambda$, for which Theorem \ref{thm:lasso-main-event} and Corollaries \ref{cor:lasso-deviation} and \ref{cor:lasso-E} hold is given in \eqref{eq:tuning-lambda}; in particular it depends on $s$.
    For this choice of $\lambda$, the prediction risk and the $\ell_2$ risk of the Lasso estimator attain the non-asymptotic minimax optimal rate $(s/n)\log (p/s)$.
    However, the knowledge of the sparsity index $s$ is needed to achieve this, which raises a problem of adaptation to sparsity~$s$.
    In Section \ref{sec:lepski}, we propose a data-driven Lasso estimator, independent of $s$, solving this adaptation problem, for which we prove essentially the same results as above.
    The argument there uses Corollary \ref{cor:lasso-deviation} as a building block.
\end{remark}

\vspace{-5mm}

\begin{remark}
    Since the assumptions on tuning parameter $\lambda$ in Theorem \ref{thm:lasso-main-event} and Corollaries \ref{cor:lasso-deviation} and \ref{cor:lasso-E}
    are given by inequalities, the case of $\lambda$ defined with $\log(2ep)$ instead of $\log(2ep/s)$ is also covered.
    With such a choice of $\lambda$, the estimators do not depend on $s$ and the results take the same form as in the standard Lasso framework, cf. \cite{MR2533469}, in which  the prediction risk and the $\ell_2$ risk achieve the suboptimal rate $(s\log p)/n$.
    However, even in this case,  Theorem \ref{thm:lasso-main-event} brings in some novelty.
    Indeed, to our knowledge, bounds in expectation, cf. Corollary \ref{cor:lasso-E},
    or in probability with arbitrary $\delta_0\in(0,1)$, cf.  Theorem \ref{thm:lasso-main-event}, were not available.
    The previous work provided only bounds in probability for fixed $\delta_0$ proportional to $1/p^c$ for an absolute constant $c>0$, in the spirit of \eqref{oi-previous}. Such bounds do not allow for control of the moments of the estimation and prediction errors without imposing extra assumptions.
    To our understanding, there was no way to fix this problem within the old proof techniques.
    On the contrary, bounds for the moments of any order can be readily derived from  Theorem \ref{thm:lasso-main-event}.
\end{remark}

\vspace{-6mm}

{\begin{remark}
In this section, the variance $\sigma$ was supposed to be known. The case of unknown $\sigma$ can be treated in a standard way as described, for example, in \cite{giraud2014introduction}. Namely, we replace $\sigma$ in \eqref{eq:tuning-lambda} by a suitable statistic $\hat \sigma$. For example, it can be shown that under the $RE$ condition, the scaled Lasso estimator $\hat \sigma^S$ is such that  $\sigma/2\le \hat \sigma^S\le 2\sigma$ with high probability provided that $s\le cn$ for some constant $c>0$, cf.  \cite[Sections 5.4 and 5.6.2]{giraud2014introduction}. Then, replacing $\sigma$ by  $\hat \sigma\triangleq 2\hat \sigma^S$ in the expression for $\lambda$, cf. \eqref{eq:tuning-lambda}, we obtain that under the same mild conditions, Corollary \ref{cor:lasso-deviation} remains valid with this choice of $\lambda$ independent of $\sigma$, up to a change in numerical constants. This remark also applies to upper bounds in probability obtained in the next sections.
\end{remark}
}

\vspace{-3mm}


\section{Aggregated Lasso estimator and adaptation to sparsity}
\label{sec:lepski}

In this section, we assume that $\vf=\design\vbeta^*$, so that we have a linear regression model
\begin{equation}\label{model0}
\vy = \design\vbeta^* + \vxi.
\end{equation}
We also assume that $\vbeta^*\in B_0(s)=\{\vbeta^*\in\R^p: \norma{\vbeta^*}_0 \le s\}$, where $s\le s_*$.
Here, $s_*\in [1,p/2]$ is a given integer. Our aim is to construct an adaptive to $s$ estimator $ \boldsymbol{\tilde\beta}$ whose prediction risk attains the optimal rate $(s/n)\log(p/s)$ simultaneously on the classes $B_0(s), 1\le s\le s_*$.
This will be done by aggregating a small number of Lasso estimators using a Lepski type procedure. The resulting adaptive estimator $ \boldsymbol{\tilde\beta}$ is computed in polynomial time and its computational complexity exceeds  that of the Lasso only by a $\log_2 p$ factor. Furthermore, we propose an estimator $\hat s$ such that the bound $\hat s \le s$ holds with high probability without the beta-min condition and without any strong assumptions on the matrix $\design$ such as the irrepresentability condition.

We denote by
$\hbeta_s$ the Lasso estimator with tuning parameter
\begin{equation*}
    \lambda(s) = {2(4+\sqrt 2)} \sigma \sqrt{\frac{\log(2ep/s)}{n}},
\end{equation*}
and we set for brevity $\theta_* = \theta(2s_*,7)$. {We will assume that $\theta_*>0$. Then, $\theta^2(s, 7)\ge \theta_*>0$, $s=1, \dots,2s_*$.}  It follows from Corollary~\ref{cor:lasso-deviation} that for any $s=1, \dots,2s_*$
\begin{align}\label{from_corollary}
\sup_{\vbeta^*\in B_0(s)}\Pro_{\vbeta^*} \Bigg(
    \norm{\design(\hbeta_s - \vbeta^*)}_n
    &\ge C_0 \sigma \sqrt{\frac{s\log({2}ep/s)}{n}}
\Bigg) \le {\frac 1 2 } \Big(\frac{s}{2ep}\Big)^{\frac{s}{\theta^2(s, 7)}}\leq \left(\frac s p\right)^s
\end{align}
 where
 \begin{equation}\label{C0}
     C_0 = {7(4+\sqrt 2)/(2\theta_*)},
 \end{equation}
and $\Pro_{\vbeta^*}$ is the probability measure associated to the model \eqref{model0}.  Furthermore, since $\lambda(s) \ge \lambda(2s)$ we also have for any $s=1, \dots,s_*$
  \begin{align}\label{from_corollary_times2}
\sup_{\vbeta^*\in B_0(2s)}\Pro_{\vbeta^*} \Bigg(
                        \norm{\design(\hbeta_s - \vbeta^*)}_n
                        &\ge \sqrt{2}C_0 \sigma \sqrt{\frac{s\log({2}ep/s)}{n}}
                                           \Bigg) \le \left(\frac {2s} p\right)^{2s}.
\end{align}
Set $b_j=2^{j-1}, \ j\in \mathbb{N}$. In what follows, we assume w.l.o.g. that $s_*\ge 2$ since the problem of adaptation does not arise for $s_*=1$. Then,  the integer $M\triangleq \max\{m\in \mathbb{N}: b_m\le s_*\}$ satisfies $M\ge 2$. Note also that $M\le \log_2(p)$. We now construct a data-driven selector from the set of estimators $\{\hbeta_{b_m}, m=2,\dots, M\}$. {We select the index $m$ as follows:
\begin{equation}\label{eq:selector}
\hat m \triangleq \min\left\{m\in \{2,\dots, M\}: \, d(\hbeta_{b_k}, \hbeta_{b_{k-1}}) \le 2w(b_k) \ \text{for all} \ k\ge m\right\}
\end{equation}
with the convention that $\hat m=M$ if the set in the above display is empty. Here, $d(\vbeta, \vbeta')=\norm{\design(\vbeta - \vbeta')}_n$ for all $\vbeta, \vbeta' \in \R^p$, and $
    w(b)=C_0\sigma\sqrt{\frac{b\log(2ep/b)}{n}}, \ \forall b\in[1,p].
    $
    }
Next, we define adaptive estimators of $s$ and $\vbeta^*$ as follows:
\begin{equation}\label{eq:selector2}
 \hat{s}=b_{\hat m-1}= 2^{\hat m-1},\qquad\text{and} \qquad
 \boldsymbol{\tilde\beta}= \hbeta_{2\hat s}.
\end{equation}
\vspace{-8mm}
\begin{theorem}\label{th:lepski}
Let $s_*\in \{2,\dots,p\}$ be such that $\theta_*>0$ and $s_*\le p/(2e)$. Then there exists an absolute constant $C_1>0$ such that, for $C_0$ given in \eqref{C0} and all $s=1, \dots,s_*$,
\begin{align}\label{th:lepski:eq1}
 \sup_{\vbeta^*\in B_0(s)}\Pro_{\vbeta^*} \Bigg(
    \norm{\design(\boldsymbol{\tilde\beta} - \vbeta^*)}_n
    &\ge C_1 C_0\sigma \sqrt{\frac{s\log({2}ep/s)}{n}}
    \Bigg) \le \big(2(\log_2(p))^2+1) \left(\frac{2s}{p}\right)^{2s}
\end{align}
and
\begin{equation}\label{th:lepski:s}
\sup_{\vbeta^*\in B_0(s)}\Pro_{\vbeta^*} (\hat{s} \le s)\ge 1-2(\log_2(p))^2 \left(\frac{2s}{p}\right)^{2s}.
\end{equation}
\end{theorem}
%


The proof of this theorem  is given in Appendix \ref{sec:proof-lepski}. {It uses only the properties \eqref{from_corollary} and \eqref{from_corollary_times2} of the family of estimators $\{\hbeta_{s}\}$.} Theorem~\ref{th:lepski} shows that the prediction error of $\boldsymbol{\tilde\beta}= \hbeta_{2\hat s}$ achieves the optimal rate of order $(s/n) \log(p/s)$.
Without \eqref{th:lepski:eq1}, the fact that estimator $\hat s$ satisfies \eqref{th:lepski:s} is not interesting.
Indeed, the dummy estimator $\tilde s = 0$ satisfies $\mathbb P(\tilde s \le s)  = 1$.
The estimator $\hat s$ is of interest because of the conjunction of \eqref{th:lepski:eq1} and \eqref{th:lepski:s};
not only it satisfies $\hat s\le s$
with high probability (cf. \eqref{th:lepski:s}) but also
 the Lasso estimator with the tuning parameter
$\lambda = {2(4+\sqrt 2)} \sigma \sqrt{\log(ep/\hat s)/n}$
achieves the optimal rate (cf. \eqref{th:lepski:eq1}).
Theorem \ref{th:lepski} shows that
this choice of the tuning parameter improves upon the prediction
bounds  for the Lasso estimator \cite{MR2533469} with the universal tuning parameter of order
$\sigma\sqrt{(\log p)/n}$.

{A procedure of the same type} is adaptive to the sparsity when measuring the accuracy by the $\ell_q$ estimation error for $1\leq q\leq2$. In this case, the risk bound is proved in the same way as in Theorem~\ref{th:lepski} due to the following observations.
First, it follows from Corollary~\ref{cor:lasso-deviation}  that for all $s=1, \dots,2s_*$, {$1\le q\le 2$,}
\begin{align}\label{from_corollary_times3}
\sup_{\vbeta^*\in B_0(s)}\Pro_{\vbeta^*} \Bigg(
                        \norma{\hbeta_s - \vbeta^*}_q
                        &\ge {C_0'} \sigma s^{1/q} \sqrt{\frac{\log({2}ep/s)}{n}}
                                           \Bigg) \le \left(\frac s p\right)^s,
\end{align}
 where {$C_0'=49(4+\sqrt 2)/(4\theta_*)$}.
Second, analogously to \eqref{from_corollary_times2}, we have for all $s=1, \dots,s_*$, {$1\le q\le 2$,}
  \begin{align}\label{from_corollary_times4}
\sup_{\vbeta^*\in B_0(2s)}\Pro_{\vbeta^*} \Bigg(
                        \norma{\hbeta_s - \vbeta^*}_q
                        &\ge {2}C_0' \sigma s^{1/q} \sqrt{\frac{\log({2}ep/s)}{n}}
                                           \Bigg) \le \left(\frac {2s} p\right)^{2s}.
\end{align}
\begin{theorem}\label{th:lepski-bis}
Let $s_*\in \{2,\dots,p\}$ be such that $\theta_*>0$ and $s_*\le p/(2e)$. Let $1\leq q\leq2$ {and let $\hat m$ be defined by \eqref{eq:selector} with $d(\vbeta, \vbeta')= \norma{\vbeta - \vbeta'}_q$ for all $\vbeta, \vbeta' \in \R^p$, and $
    w(b)=C_0\sigma b^{1/q}\sqrt{\frac{\log(2ep/b)}{n}}, \ \forall b\in[1,p],
    $
    where  {$C_0=49(4+\sqrt 2)/(4\theta_*)$}. Let
$\boldsymbol{\tilde\beta}$ be defined as in \eqref{eq:selector2}.
} Then, there exists an absolute constant $C_1>0$ such that, for all $s=1, \dots,s_*$,
\begin{equation*}
 \sup_{\vbeta^*\in B_0(s)}\Pro_{\vbeta^*} \Bigg(
    \norma{\boldsymbol{\tilde\beta} - \vbeta^*}_q
    \ge C_1 C_0\sigma s^{1/q} \sqrt{\frac{\log({2}ep/s)}{n}}
    \Bigg) \le \big(2 (\log_2(p))^2+1\big)\Big(\frac{2s}{p}\Big)^{2s}.
\end{equation*}
\end{theorem}
The proof of this theorem is given in Appendix \ref{sec:proof-lepski}. Due to \eqref{from_corollary_times3}  and  \eqref{from_corollary_times4}, it  is quite analogous to the proof of Theorem~\ref{th:lepski}. 

\section{Optimal rates for the Slope estimator}
\label{sec:slope}

In this section, we study the Slope estimator with weights $\lambda_j$ given in  \eqref{eq:recommendation-weights-slope}.
We will use the following assumption on the design matrix $\design$
that we call the \emph{Weighted Restricted Eigenvalue} condition,
or shortly the $WRE$ condition. Let $c_0>0$, $s\in\{1,\dots,p\}$ be constants, and let
$\vlambda = (\lambda_1,\dots,\lambda_p) \in \R^p$
be a vector of weights not all equal to 0 such that
$\lambda_1\ge \lambda_2\ge\dots\ge \lambda_p\ge0$.

\begin{condition}[$WRE(s,c_0)$]
    \label{assumption:slope-re}
    The design matrix $\design$ satisfies $\norm{\design\ve_j}_n \le 1$ for all $j=1,\dots,p$
    and
    \begin{equation*}
        \vartheta(s,c_0) \triangleq \min_{\vdelta\in\cC_{WRE}(s,c_0):\vdelta\ne\vzero} \frac{\norm{\design\vdelta}_n}{\norma{\vdelta}_2}
        > 0,
    \end{equation*}
    where
    $
    \cC_{WRE} (s,c_0) \triangleq \left\{
            \vdelta\in\R^p: |\vdelta|_*\le (1+c_0) \norma{\vdelta}_2 \sqrt{\sum_{j=1}^s\lambda_j^2}
    \right\}
    $
    is a cone in $\R^p$.
\end{condition}

This condition is stated for any weights $\lambda_1\ge\dots\ge\lambda_p\ge0$ but we will use it only for $\lambda_j$ given in  \eqref{eq:recommendation-weights-slope} and in that case the cone is equivalently defined as
$$
    \cC_{WRE} (s,c_0) = \left\{
            \vdelta\in\R^p: \sum_{j=1}^p \delta_j^\sharp \sqrt{\log(2p/j)}\le (1+c_0) \norma{\vdelta}_2 \sqrt{\sum_{j=1}^s\log(2p/j)}
    \right\}
.
$$
Let us compare the $WRE$ condition with the $SRE$ condition.
Assume that $\vdelta$ belongs to the cone $\cC_{SRE}(s,c_0)$, that is,
$\norma{\vdelta}_1 \le (1+c_0) \sqrt s \norma{\vdelta}_2$. Then also
$
    \sum_{j=s+1}^p \delta_j^\sharp
    \le
    {(1+c_0)} \sqrt s \norma{\vdelta}_2
$, and we have
\begin{eqnarray*}
    \sum_{j=s+1}^p\delta_j^\sharp \sqrt{\log(2p/j)}
    &\le &
    \sqrt{\log(2p/s)}
    \sum_{j=s+1}^p \delta_j^\sharp
    \le
     {(1+c_0)} \norma{\vdelta}_2 \sqrt{s\log(2p/s)} \\
     &\le&
      {(1+c_0)} \norma{\vdelta}_2 \sqrt{\sum_{j=1}^s \log(2p/j)},
\end{eqnarray*}
where the last inequality follows from \eqref{eq:algebra_stirling}. For the first $s$ components, the Cauchy-Schwarz inequality
yields
\begin{equation*}
    \sum_{j=1}^s\delta_j^\sharp \sqrt{\log(2p/j)}
    \le
    \norma{\vdelta}_2 \sqrt{\sum_{j=1}^s \log(2p/j)}.
\end{equation*}
Combining the last two displays we find that $\vdelta\in \cC_{WRE}(s, {1+c_0})$. Thus,
$\cC_{SRE}(s,c_0)\subseteq \cC_{WRE}(s, {1+c_0})$, {so that
the $WRE(s, 1+c_0)$ condition
implies the $SRE(s,c_0)$ condition.}
A more detailed comparison between these two conditions
as well as examples of random matrices, for which both conditions hold
are given in Section \ref{sec:assumptions}.
We are now ready to state our main result on the Slope estimator.

\begin{thm}
    \label{th:slope}
    Let $s\in\{1,\dots,p\}$, $\gamma\in (0,1)$ and $\tau\in[0,1 - \gamma)$.
    Set
    $c_0 = c_0(\gamma,\tau) = \frac{ 1+ \gamma  + \tau}{1 - \gamma - \tau}$. Let the tuning parameters $\lambda_j$ be     defined by
    \eqref{eq:recommendation-weights-slope} with constant
    \begin{equation}
        {A\ge (4+\sqrt 2)/\gamma}.
        \label{eq:condition-A}
    \end{equation}
    Let $\delta_0\in(0,1)$.
    Then, on the event \eqref{eq:main-event},
    the Slope estimator $\hbeta$ that minimizes \eqref{eq:def-hbeta} with the weights $\lambda_1,\dots,\lambda_p$
    satisfies
    \begin{equation}
        2\tau \norma{\hbeta - \vbeta}_*
        +
        \norm{\design\hbeta - \vf}^2_n
        \le
        \norm{\design\vbeta - \vf}^2_n
        + C_{\gamma,\tau}'(s,\delta_0) \sum_{j=1}^s\lambda_j^2
        \label{eq:soi-slope}
    \end{equation}
    simultaneously for all $\vf\in\R^n$, all $s=1,\dots,p$,
    and all $\vbeta\in\R^p$ such that $\norma{\vbeta}_0 = s$,
    where we set
    \begin{equation*}
        C_{\gamma,\tau}'(s,\delta_0)
        \triangleq
        (1+\gamma + \tau )^2
        \left(\frac{\log(1/\delta_0)}{s\log(2p/s)} \vee \frac1{\vartheta^2(s,c_0(\gamma,\tau))} \right)
        ,
        \qquad
        s=1,\dots,p,
    \end{equation*}
    if $WRE(s,c_0)$ holds, and  $C_{\gamma,\tau}'(s,\delta_0)=\infty$ otherwise.
    Furthermore, if $\vf = \design\vbeta^*$ for some $\vbeta^*\in\R^p$ with $\norma{\vbeta^*}_0 \le s$, then
    on the event \eqref{eq:main-event} we have
    \begin{align}
        2\tau \norma{\hbeta - \vbeta^*}_*
        &\le
        C_{\gamma,\tau}'(s,\delta_0) \sum_{j=1}^s \lambda_j^2
        ,
        \label{eq:estimation-well-specified-slope}
        \\
        \norma{\hbeta - \vbeta^*}_2
        &\le
        \frac{C_{\gamma,0}'(s,\delta_0)}{1+\gamma} \Big(\sum_{j=1}^s \lambda_j^2\Big)^{1/2}
        .
        \label{eq:esimtation-ell2-slope}
    \end{align}
\end{thm}
The proof of Theorem \ref{th:slope} is given in Section \ref{sec:proof-slope}. It follows the same route as the proof of Theorem \ref{thm:lasso-main-event}.
Since $\lambda_1,\dots,\lambda_p$ satisfy \eqref{eq:recommendation-weights-slope} then by \eqref{eq:algebra_stirling},
for all $s=1,\dots,p$ we have
\begin{equation}\label{equival}
    \frac{A^2 \sigma^2 s\log(2p/s)}{n}
    \le
    \sum_{j=1}^s \lambda_j^2
    \le
    \frac{A^2 \sigma^2 s\log(2ep/s)}{n}
\end{equation}
so that the Slope estimator achieves the optimal rate for the prediction error and the $\ell_2$-estimation error.
The presentation of Theorem \ref{th:slope} is similar to that of Theorem \ref{thm:lasso-main-event} for the Lasso, although
there are some differences that will be highlighted after the following corollaries.
Corollary \ref{cor:slope-deviations} below is an immediate consequence of
Theorems \ref{thm:main-event} and \ref{th:slope}
with $\gamma=1/2$, $\tau = 1/4$ and
$\delta_0 = \left(\frac{s}{2p}\right)^{\frac{s}{\vartheta^2(s,7)}}$
or $\tau=0$ and $\delta_0 = \left(\frac{s}{2p}\right)^{\frac{s}{\vartheta^2(s,3)}}$.

\begin{corollary}
    \label{cor:slope-deviations}
    Let $s\in\{1,\dots,p\}$. Assume that the
    $WRE(s,7)$ condition holds. Let $\hbeta$ be the Slope estimator with tuning parameters
    $\lambda_1,\dots,\lambda_p$  satisfying
    \eqref{eq:recommendation-weights-slope} for $A\ge {2(4+\sqrt 2)}$.
    Then, with probability at least $1 - {\frac 1 2} \left(\frac{s}{2p}\right)^{\frac{s}{\vartheta^2(s,7)}}$,  we have
    \begin{equation*}
                   \frac 1 2 \norma{\hbeta - \vbeta}_*
            +
            \norm{\design\hbeta - \vf}^2_n
            \le
            \norm{\design\vbeta - \vf}^2_n
            + \frac{49 \sum_{j=1}^s\lambda_j^2 }{16 \vartheta^2(s,7)}
          \label{eq:slope-explicit-constants-soi}
    \end{equation*}
    for all $\vbeta\in\R^p$ such that $\norma{\vbeta}_0\le s$, and all $\vf\in\R^n$.
    Furthermore, if $\vf = \design\vbeta^*$ for some $\vbeta^*\in\R^p$ with $\norma{\vbeta^*}_0 \le s$
    then
    \begin{align*}
        \mathbb P\left(
            \norma{\hbeta-\vbeta^*}_2^2
            \le \frac{9 \sum_{j=1}^s\lambda_j^2}{4\vartheta^4(s,3)}
        \right)
        \ge
        1- {\frac 1 2} \left(\frac{s}{2p}\right)^{\frac{s}{\vartheta^2(s,3)}}.
        \label{eq:slope-explicit-constants-estimation-ell2}
    \end{align*}
\end{corollary}
The fact that Theorems \ref{thm:main-event} and \ref{th:slope} hold for any $\delta_0\in(0,1)$
allows us to integrate the bounds \eqref{eq:soi-slope} and \eqref{eq:esimtation-ell2-slope}
to obtain the following oracle inequalities and bounds on the estimation error in expectation.

\begin{corollary}
    \label{cor:slope-E}
     Let $s\in\{1,\dots,p\}$.
    Assume that the
    $WRE(s,7)$ condition holds.
    Let $\hbeta$ be the Slope estimator with tuning parameters
    $\lambda_1,\dots,\lambda_p$  satisfying
    \eqref{eq:recommendation-weights-slope} for $A\ge { 2(4+\sqrt 2) }$.
      Then,
          \begin{equation*}
        \E \left[ \frac 1 2 \norma{\hbeta - \vbeta}_*
                +
                \norm{\design\hbeta - \vf}^2_n
        \right]
        \le
        \norm{\design\vbeta - \vf}^2_n
        + \frac{49\sum_{j=1}^s\lambda_j^2}{16} \left(
            \frac{1}{\vartheta^2(s,7)}
            + \frac{1}{{2}\log(2p)}
        \right)
    \end{equation*}
    for all $\vbeta\in\R^p$ such that $\norma{\vbeta}_0 \le s$, and all $\vf\in\R^n$.
    Furthermore, if $\vf=\design\vbeta^*$
    for some $\vbeta^*\in\R^p$ with $\norma{\vbeta^*}_0 \le s$,
    then
    \begin{equation*}
        \E \left[ \norma{\hbeta - \vbeta^*}_2^2 \right]
        \le
        \frac{9 \sum_{j=1}^s\lambda_j^2 }{4}
        \left(
            \frac{1}{\vartheta^4(s,3)}
            +
            \frac{{1}}{(\log(2p))^2}
        \right)
        .
    \end{equation*}
\end{corollary}

{Since $\hbeta$ does not depend on $s$, the first inequality in Corollary \ref{cor:slope-E} and \eqref{equival} imply a ``balanced'' oracle inequality:
\begin{equation}\label{eq:6.5}
\E \left[
                \norm{\design\hbeta - \vf}^2_n
        \right]
        \le
       \inf_{\vbeta\in \R^p} \left[\norm{\design\vbeta - \vf}^2_n
        +
            C(\norma{\vbeta}_0)\frac{\norma{\vbeta}_0}{n}\log\Big(\frac{2ep}{\norma{\vbeta}_0\vee 1}\Big)       \right],
            \quad \forall \vf\in\R^n,
\end{equation}
where
$$
C(\norma{\vbeta}_0)= \frac{49A^2\sigma^2}{16} \left(
    \frac{1}{\vartheta^2(\norma{\vbeta}_0,7)}+\frac{1}{{2}\log(2p)} \right)
$$
if $\vartheta^2(\norma{\vbeta}_0,7)\ne 0$ and $C(\norma{\vbeta}_0)=\infty$ otherwise.} {This formulation might be of interest in the context of aggregation as explained, for example, in \cite{rigollet2011exponential}.}

Corollaries \ref{cor:slope-deviations} and \ref{cor:slope-E} are the analogs of
Corollaries  \ref{cor:lasso-deviation} and \ref{cor:lasso-E} for the Lasso.
The proof of Corollary \ref{cor:slope-E} is omitted. It is deduced from Theorem \ref{th:slope} exactly
in the same way as Corollary \ref{cor:lasso-E} is deduced from Theorem \ref{thm:lasso-main-event}  in Section \ref{sec:proof-lasso}.

The results in Section \ref{sec:lasso-4} and in the present section
show that both the Lasso estimator with tuning parameter of order $\sigma\sqrt{\log(p/s)/n}$
and the Slope estimator with weights \eqref{eq:recommendation-weights-slope}
achieve the optimal rate $(s/n)\log(p/s)$ for the $\ell_2$-estimation and the prediction error.
We now highlight some differences between these results on Slope and Lasso.

The first difference, in favor of Slope, is that
Slope achieves the optimal rate  adaptively to the unknown sparsity $s$.
This was previously established in \cite{MR3485953}
for random design matrices $\design$ with i.i.d. $\mathcal N(0,1)$ entries
and in \cite{LM_reg_sparse}
for random $\design$ with independent subgaussian isotropically distributed rows.
The results of the present section show that, in reasonable generality, Slope achieves rate optimality for deterministic
design matrices.  Namely, it is enough to check a rather general condition $WRE$, which is only slightly more constraining
than the $RE$ condition commonly used in the context of Lasso.   It is also shown in Section \ref{sec:assumptions} that  the $WRE$ condition holds with high probability for a large class of random design matrices.
This includes design matrices with
i.i.d. anisotropically distributed rows, for example, matrices with i.i.d. rows
distributed as $\mathcal N(\vzero, \Sigma)$ where $\Sigma\in\R^{p\times p}$ is not invertible.

The second difference is that our results for the Lasso are obtained in greater generality than for the Slope.
Indeed, the $SRE$ condition required in Section \ref{sec:lasso-4} for the Lasso
is weaker that the $WRE$ condition required here for the Slope. We refer to
Section \ref{sec:assumptions} for a more detailed comparison of these conditions.
Furthermore, for $1\le q<2$, in Section-\ref{sec:lasso-4} we obtain rate optimal bounds on the $\ell_q$-errors of the Lasso estimator, while for its Slope counterpart we can only control the rate in the $|\cdot|_*$ and $\ell_2$ norms, cf. \eqref{eq:estimation-well-specified-slope} and \eqref{eq:esimtation-ell2-slope} (and all the interpolation norms in between but those are not classical to measure statistical performances). Of course, for the weights \eqref{eq:recommendation-weights-slope}, the trivial relation $|\vbeta|_*\ge C\sigma |\vbeta|_1/\sqrt{n}$ holds, where $C>0$ is a constant. This and \eqref{eq:estimation-well-specified-slope} lead to a bound on the $\ell_1$-error of the Slope estimator,   which is however suboptimal. The same problem arises with the $\ell_q$-norms with $1<q<2$ if the bounds are obtained by interpolation between such a suboptimal bound for the $\ell_1$-error and the $\ell_2$ bound \eqref{eq:esimtation-ell2-slope}.

Finally, note that the aggregation scheme proposed in \Cref{sec:lepski} requires the knowledge of a lower bound  $\theta_*$ on the $SRE$ constants $\theta(\cdot,7)$
in order to adaptively achieve the optimal rate $s\log(p/s)/n$. If $n>cs_*\log(p/s_*)$  for some numerical constant $c>0$, and the regressors are random and satisfy suitable assumptions (see Theorem \ref{thm:design-random-matrices-slope}), one can compute the numerical constant $\theta_*$ that gives the required bound with high probability.  It can be also computed when the matrix $\design$ is deterministic and satisfies the mutual coherence condition. But in general case
it is hard to compute such $\theta_*$.
On the other hand, the Slope estimator adaptively
achieves the optimal rate $s\log(p/s)/n$ without the knowledge
of any $RE$-type constants.


\section{Minimax lower bounds}
\label{sec:lower}

In this section, we provide the minimax lower bounds for the prediction risk and $\ell_q$-estimation risk
on the class $B_0(s)$.
Several papers have addressed this issue for the prediction risk \cite{rigollet2011exponential,lptv2011,raskutti2011,abramovich2010,verzelen2012}, for the $\ell_2$-estimation risk \cite{raskutti2011,verzelen2012,candes_davenport2013}, and for the $\ell_q$-estimation risk with general $q$ \cite{raskutti2011,ye_zhang2010,lptv2011}. We are interested here in non-asymptotic bounds and therefore the results in \cite{raskutti2011,ye_zhang2010} obtained in some asymptotics do not fit in our context. Another issue is that the papers cited above, except for \cite{lptv2011}, deal with lower bounds for the expected squared risk or power risk \cite{raskutti2011,ye_zhang2010} and thus cannot be used to match the upper bounds in probability that are in the focus of our study
in this paper. The only result that can be applied in this context is Theorem 6.1 in \cite{lptv2011}. It gives a non-asymptotic lower bound for general loss functions under the condition that the ratio of minimal and maximal $2s$-sparse eigenvalue of the Gram matrix $\design^T\design/n$ is bounded from below by a constant.  It matches our upper bounds both for the prediction risk and for the $\ell_q$-estimation risk. Note that Theorem 6.1 in \cite{lptv2011} deals with group sparsity and is therefore more general than in our setting.  Thus, we only refer to the case $T=1$ of Theorem 6.1 in \cite{lptv2011} corresponding to ordinary sparsity. Here, we provide an improvement on it, in the sense that for the lower bound in $\ell_q$, we drop the ratio of sparse eigenvalues condition. Thus, in the next theorem the lower bound for the $\ell_q$-estimation risk holds  for any design matrix $\design$.  For the prediction risk, the lower bound that we state below is borrowed from \cite[Theorem 6.1]{lptv2011} and it is meaningful only if the minimal sparse eigenvalue is positive.
For any matrix $\design\in\R^{n\times p}$ and any $s\in [1,p]$, define the minimal and maximal $s$-sparse eigenvalues
as follows:
\begin{equation}\label{SPD}
\bar\theta_{\min}(\design, s) \triangleq \min_{\vdelta\in B_0(s)\setminus \{0\}} \frac{\norm{\design\vdelta}_n}{\norma{\vdelta}_2},
 \qquad
\bar\theta_{\max}(\design, s) \triangleq \max_{\vdelta\in B_0(s)\setminus \{0\}} \frac{\norm{\design\vdelta}_n}{\norma{\vdelta}_2}.
\end{equation}
In particular, $\bar\theta_{\max}(\design, 1)=\max_{j=1,\dots,p}\norm{\design \ve_j}_n. $

Let $\ell\,:\, \R_{+} \rightarrow \R_{+}$ be a nondecreasing function such that $\ell(0)= 0$ and $\ell
\not\equiv 0$. Define
$$
\psi_{n,q} =\sigma s^{1/q} \sqrt{\frac{\log (ep/s)}{n}}, \quad \quad 1\le q\le
\infty,
$$
where we set $s^{1/\infty}\triangleq 1$. Let ${\mathbb
 E}_{\vbeta}$ denote the expectation with respect to the measure ${\mathbb
 P}_{\vbeta}$.

\begin{theorem} \label{th_minimax_lb} Let $p \geq 2$, $s\in[1, p/2]$,
$n \geq 1$ be integers, and  let $1\le q \le\infty$. Assume that $\vf=\design\vbeta^*$ and $\vxi\sim \cN(\vzero,\sigma^2 I_{n\times n})$, $\sigma>0$.  Then the following holds.
\begin{itemize}
\item[(i)] There exist positive constants
${\bar b}, {\bar c}$ depending only on $\ell(\cdot)$ and $q$ such
that
\begin{equation}\label{eq:minm_lb2}
 \inf_{\hat\vtau} \inf_{\design}\sup_{\vbeta^*\in B_0(s)}{\mathbb
 E}_{\vbeta^*} \ell\left({\bar b}\psi_{n,q}^{-1}\bar\theta_{\max}(\design, 1)
 \norma{\hat \vtau - \vbeta^*}_q \right)
  \geq {\bar c}
 \end{equation}
where $\inf_{\hat\vtau}$ denotes the infimum over all estimators ${\hat\vtau}$
of $\vbeta^*$, and $\displaystyle{\inf_{\design}}$ denotes the infimum over all matrices $\design\in\R^{n\times p}$.
\item[(ii)]
There exist positive constants ${\bar b}, {\bar c}$
depending only on $\ell(\cdot)$ such that
  \begin{equation}\label{eq:minm_lb1}
 \inf_{\hat\vtau}\inf_{\design}\sup_{\vbeta^*\in B_0(s)}{\mathbb
 E}_{\vbeta^*}  \ell\left({\bar b}\psi_{n,2}^{-1}\frac{\bar\theta_{\max}(\design,2s)}{\bar\theta_{\min}(\design,2s)}\norm{\design(\hat\vtau - \vbeta^*)}_n\right)
  \geq {\bar c}
 \end{equation}
 where, by definition, the expression under the expectation is  $+\infty$ for matrices $\design$ such that $\bar\theta_{\min}(\design,2s)=0$.
 \end{itemize}
\end{theorem}

\begin{proof}
Part (ii) is a special case of Theorem 6.1 in \cite{lptv2011} corresponding to the number of groups $T=1$.
To prove part (i), we use Lemma~\ref{lem:verzelen} stated in the Appendix. Let $\mathcal{B} = \{{\vbeta}= a{\vomega}: \vomega\in \Omega\}$, where $\Omega$ is a subset of $\big\{1, 0,-1\big\}^p$ satisfying Lemma~\ref{lem:verzelen},
$a= \alpha \bar\theta_{\max}^{-1}(\design, 1) \psi_{n,q}s^{-1/q}$, and $0<\alpha< {\tilde c}^{1/2}/4$, where $\tilde{c}$ is a constant appearing in Lemma~\ref{lem:verzelen}. It follows from Lemma~\ref{lem:verzelen} that $\mathcal{B}\subset B_0(s)$, and
\begin{eqnarray}\label{verz1}
    \norma{\vbeta - \vbeta'}_q \ge 4^{-1/q}\alpha \bar\theta_{\max}^{-1}(\design, 1)  \psi_{n,q}
 \end{eqnarray}
 for any two distinct elements $\vbeta$ and $\vbeta'$ of $\mathcal{B}$. Again from Lemma~\ref{lem:verzelen}, for all $\vbeta$ and $\vbeta'$ in $\mathcal{B}$, the Kullback-Leibler divergence $\mathcal{K}(\Pro_{\vbeta},\Pro_{\vbeta'})$ between the probability measures $\Pro_{\vbeta}$ and $\Pro_{\vbeta'}$ satisfies
\begin{eqnarray}\label{verz2}
   \mathcal{K}(\Pro_{\vbeta},\Pro_{\vbeta'}) &=& \frac{n}{2\sigma^2}     \norm{\design(\vbeta-\vbeta')}^2_n \le \frac{\alpha^2n}{\sigma^2} \psi_{n,q}^{2}s^{1-2/q} \\ 
    &= & \frac{\alpha^2 s}{n}\log\left(\frac{ep}{s}\right) \le  \frac{\tilde c s}{16 n}\log\left(\frac{ep}{s}\right) \le \frac{1}{16}\log ({\rm Card} \,\mathcal{B}). \nonumber
              \end{eqnarray}
 The bound \eqref{eq:minm_lb2} now follows from   \eqref{verz1} and   \eqref{verz2}  in view of \cite[Theorem~2.7]{tsy09}.
\end{proof}

As a consequence of Theorem~\ref{th_minimax_lb}, we get, for
example, lower bounds for the squared loss $\ell(u)=u^2$ and for
the indicator loss $\ell(u)=I\{u\ge 1\}$. The indicator loss is
relevant for comparison with the upper bounds in probability obtained in the previous sections.
For example,
Theorem~\ref{th_minimax_lb} with this loss and $1\le q\le 2$ implies that
 for any estimator
 $\hat\vtau$ there exists $\vbeta^*\in B_0(s)$ such that, with ${\mathbb
 P}_{\vbeta^*}$-probability at least $\bar{c}>0$,
$$
\norm{\design(\hat\vtau - \vbeta^*)}_n^2 \ge C\frac{\sigma^2 s}{n}\log\left(\frac{ep}{s}\right)
$$
and
$$
\norma{\hat\vtau - \vbeta^*}_q \ge C\frac{\sigma s^{1/q}}{\sqrt{n}}\log\left(\frac{ep}{s}\right)
^{1/2}, \quad 1\le q\le 2,
$$
where $\bar{c}>0$ is a numerical constant and $C>0$ is some constant depending only on $\design$. The rates on the right-hand side of these
inequalities have the same form as in the corresponding upper
bounds for the Lasso and Slope estimators obtained in Corollaries \ref{cor:lasso-deviation} and \ref{cor:slope-deviations}.
The fact that the constants $C$ here depend on the design implies that the optimality is not guaranteed for all configurations of $n, s,p$. Thus, we get the rate optimality under the assumption that $s\log(ep/s) < cn$ for the $\ell_2$-risk, and under the assumption $s\log(ep/s) < c R$ for the prediction risk, where $c>0$ is a constant. Here, $R$ denotes the rank of matrix $\design$. Concerning the prediction risk, this remark is based on the following fact.

\begin{corollary}\label{cor:verz}
Let $p \geq 2$, $s\in[1, p/2]$ and
$n \geq 1$ be integers. If for some matrix $\design\in\R^{n\times p}$ and some $b>0$ we have $\frac{\bar\theta_{\max}(\design,2s)}{\bar\theta_{\min}(\design,2s)}\le b$, then there exists $c=c(b)>0$ such that $s\log(ep/s) < c R$.
\end{corollary}

This corollary follows immediately from \eqref{eq:minm_lb1} with $\ell(u)=u^2$ and the fact that the minimax expected squared risk is bounded from above by $\sigma^2 R/n$ (cf., for example, \cite{rigollet2011exponential}).

In view of Corollary~\ref{cor:verz}, the bound \eqref{eq:minm_lb1} is non-trivial only when $s\log(ep/s) < c R$. The bound \eqref{eq:minm_lb2} does not have such a restriction and remains non-trivial for all $n, s,p$. However, it is known that for $q=2$ and $s\log(ep/s)\gg n$ this bound is not optimal \cite{verzelen2012}. Anyway, \eqref{eq:minm_lb2} shows that if $s\log(ep/s)\gg n$, the $\ell_2$-risk diverges, so this case is of minor interest.  Also note that the upper bounds of Theorems~\ref{thm:lasso-main-event}, \ref{th:slope} and their corollaries rely on $RE$ type conditions, and we need that $s\log(ep/s) < cn$ for these conditions to be satisfied if the matrix $\design$ is random (see more details in the next section).


\section{Assumptions on the design matrix}
\label{sec:assumptions}

\subsection{Equivalence between $RE$, $SRE$ and $s$-sparse eigenvalue conditions}\label{subsec:8.1}
Along with the $RE$ and $SRE$ conditions
defined in Section \ref{sec:lasso-4} we consider here the
$s$-sparse eigenvalue condition defined
as follows, for any $s\in\{1,\dots,p\}$.

\begin{condition}[$s$-sparse eigenvalue]
    \label{def:weak_rip}
    The design matrix $\design$ satisfies $\norm{\design\ve_j}_n \le 1$ for all $j=1,\dots,p$,
    and $ \bar\theta_{\min}(\design, s)>0$.
\end{condition}
In this section, we will write for brevity $ \bar\theta_{\min}(s) = \bar\theta_{\min}(\design, s)$. The next proposition
establishes the equivalence between the three conditions mentioned above.

\begin{proposition}
    Let $c_0>0$ and $s\in\{1,\dots,p\}$.
    We have the following implications.
    \begin{itemize}
        \item[{\rm (i)}] If condition $SRE(s,c_0)$ holds then
            condition $RE(s,c_0)$ holds and $\kappa(s,c_0) \ge \theta(s,c_0)$.
        \item[{\rm (ii)}] If condition $RE(s,c_0)$ holds then
            the $s$-sparse eigenvalue condition holds
            and $\bar\theta_{\min}(s) \ge \kappa(s,c_0)$.
        \item[{\rm (iii)}]
            Let $\theta_1>0$. 
            If the $s$-sparse eigenvalue condition holds
            with $\bar\theta_{\min}(s) \ge \theta_1$ then
            the $SRE(s_1,c_0)$ condition holds and
            $\theta(s_1,c_0) \ge \theta_1 / \sqrt 2$ for $s_1\le (s-1)\theta_1^2/(2c_0^2)$.
    \end{itemize}
\end{proposition}
\begin{proof} Part (i) follows from \eqref{CC}. Next,
    if $\vdelta\in B_0(s)$
    then obviously $\norma{\vdelta}_1 \le (1+c_0) \sum_{j=1}^s \delta_j^\sharp$
    with $c_0 = 0$.
    Thus, the set of all $s$-sparse vectors $B_0(s)$
    is included in the cone ${\mathcal C}_{RE}(s,c_0)$ for any $c_0>0$. This implies (ii).     To prove (iii), we use Lemma 2.7 in  \cite{LM_compressed}, which implies that if $\bar\theta_{\min}(s) \ge \theta_1$, and $\norm{\design\ve_j}_n \le 1$ for all $j=1,\dots,p$, then $\norm{\design \vdelta}_n^2\ge  \theta_1^2 |\vdelta|_2^2 - |\vdelta|_1^2/(s-1)$ for all $\vdelta\in\R^p$.
\end{proof}
The message of the above proposition is that
the three conditions -- $RE(s,c_0)$, $SRE(s,c_0)$ and
the $s$-sparse eigenvalue condition --
are equivalent up to absolute constants.
This equivalence has two main consequences for the results of the present paper.

\begin{itemize}
\item
First, the results on the Lasso in Sections \ref{sec:lasso-4} and \ref{sec:lepski} are proved under the $SRE(s,c_0)$ condition.
The above equivalence shows that, for some integer $s_1$, which is of the same order as $s$,
the oracle inequalities and the estimation bounds of Sections \ref{sec:lasso-4} and \ref{sec:lepski}
are valid under the Restricted Eigenvalue condition $RE(s_1,c_0)$.
\item
Second, the $s$-sparse eigenvalue condition is known to hold with high probability
for rather general random matrices with i.i.d. rows.
By the above equivalence,
conditions $RE(s,c_0)$ and $SRE(s,c_0)$
are satisfied for the same random matrices.
A useful sufficient condition for the $s$-sparse eigenvalue
condition is the small ball condition
\cite{Shahar-Vladimir, Shahar-COLT, MR3367000}.
A random vector $\vx$ valued in $\R^p$ is said to satisfy the small ball
condition over $B_0(s_1)$ if there exist positive numbers $u$ and $\beta$ such that
\begin{equation}\label{eq:small_ball_cond}
       \Pro\left[
        | \vdelta^T \vx|\geq u \norma{\vdelta}_2
      \right]\geq \beta, \qquad  \forall \vdelta\in B_0(s_1).
\end{equation}
Let $\design\in\R^{n\times p}$ be a matrix with i.i.d. rows that have the same distribution as $\vx$ satisfying \eqref{eq:small_ball_cond}.
Corollary~2.5 in \cite{LM_compressed} establishes that, for such $\design$ we have
$\bar\theta_{\min}(s_1)>u/\sqrt 2$
with probability at least $1-\exp(- C n \beta^2)$
if $n\geq (C'/\beta^2) s \log(ep/s)$ for some absolute constants
$C,C'>0$.
\end{itemize}

Note that condition \eqref{eq:small_ball_cond} is very mild.
For instance, a vector $\vx$ with independent components that have a Cauchy distribution
satisfies this condition.
Thus, condition \eqref{eq:small_ball_cond} is quite different in nature from any
concentration property.
On the other hand, the property $\max_{j=1,\dots,p}\norm{\design\ve_j}_n \le 1$
assumed in the above three conditions (which is usually seen as a simple normalization) requires concentration. Indeed, this inequality can be written as
\begin{equation*}
     \frac 1 n \sum_{i=1}^n (\vx_i^T \ve_j)^2 \le 1, \quad  \forall j=1,\dots,p,
\end{equation*}
where $\vx_1,\dots,\vx_n$ are the i.i.d. rows of $\design$.
We have here a sum of  i.i.d. positive random variables.
Satisfying $\max_{j=1,\dots,p}\norm{\design\ve_j}_n \le 1$ with high probability
requires $\E[ (\vx^T\ve_j)^2 ] < 1$ and some concentration property
of the random variables $(\vx^T\ve_j)^2$ for all $j=1,\dots,p$.
It is proved in  \cite{LM_compressed} that this property holds with high
probability when the components of $\vx$ (that do not have
to be independent) have
moments with a polynomial growth up to the order $\log(ep)$ and that this condition may be violated with probability greater than $1/2$ if the coordinates only have $\log(ep)/\log\log(ep)$ such moments.

In conclusion, for a large class of random matrices
with i.i.d.  rows,
condition $SRE(s,c_0)$ holds with high probability if
\begin{equation}\label{eq:snp}
s\log(ep/s)\le cn,
\end{equation}
 where $c>0$ is a constant.

\subsection{Design conditions for the Slope estimator}

Theorem \ref{th:slope} and Corollaries \ref{cor:slope-deviations}, \ref{cor:slope-E}
establish prediction and estimation bounds for the Slope estimator
under the $WRE(s,c_0)$ condition.
{It was explained in Section \ref{sec:slope} that the $WRE(s,c_0)$ condition implies the
 $SRE(s,1+c_0)$ condition.}
The converse is not true -- there is no equivalence between the two conditions. However, a simple observation leads to the following sufficient condition
for $WRE(s,c_0)$.
\begin{proposition}\label{prop:wre}
    Let $s\in\{1,\dots,p\}$, $c_0>0$, and let the weights $\lambda_j$ be defined by \eqref{eq:recommendation-weights-slope}.
    Set $s_2 = {\lceil s\log(2ep/s)/\log 2\rceil}$.
    If the $SRE(s_2,c_0)$ condition holds then the
    $WRE(s,c_0)$ condition holds, and $\vartheta(s,c_0) \ge \theta(s_2,c_0)$.
\end{proposition}
\begin{proof}
    If $\vdelta\in\cC_{WRE}(s,c_0)$, then
    \begin{equation*}
        (1+c_0) \left(\sum_{j=1}^s \lambda_j^2\right)^{1/2} \norma{\vdelta}_2
        \ge
        \norma{\vdelta}_*
        =
        \sum_{j=1}^p \lambda_j \delta_j^\sharp
        \ge
         {\lambda_p}
         \sum_{j=1}^p \delta_j^\sharp = \lambda_p \norma{\vdelta}_1.
    \end{equation*}
    This, together with \eqref{eq:recommendation-weights-slope} and \eqref{eq:algebra_stirling} imply
    $\norma{\vdelta}_1 \le ( 1 +  c_0) \sqrt{s\log(2ep/s)/\log 2} \norma{\vdelta}_2$.
    Thus, $\vdelta\in\cC_{SRE}(s_2,c_0)$.
\end{proof}

Proposition \ref{prop:wre} implies that, under the same assumptions as discussed in Subsection \ref{subsec:8.1}, for large classes of random matrices with i.i.d. rows,
condition $WRE(s,c_0)$ holds with high probability
whenever $s\log^2(ep/s)\le cn$ where $c>0$ is a constant.
This inequality on $s,p$ and $n$ differs from \eqref{eq:snp} only in  an extra logarithmic
factor.
Moreover, the next theorem shows that this extra
factor is not necessary if the row vectors of
$\design$ are sub-gaussian.

\begin{thm}
    \label{thm:design-random-matrices-slope}
    There exist absolute constants $C,C'>0$
    such that the following holds.
    Let $c_0,\kappa>0$ and let $s\in\{1,\dots,p\}$.
    Let $\vx_1,\dots,\vx_n$ be i.i.d. copies of a mean zero random variable $\vx$ valued in $\R^p$
    with covariance matrix $\Sigma = \E[\vx \vx^T]= (\Sigma_{ij})_{1\leq i,j\leq p}$.
    Let $L\ge 1$ and assume that $\vx$ is $L$-subgaussian in the sense that
    \begin{equation}
        \label{eq:sub-gauss}
                \E\exp\left(\vdelta^T \vx\right)
        \leq
        \exp\left( \frac{L^2 \norma{\Sigma^{1/2}\vdelta}_2^2}{2}\right), \quad \forall \vdelta\in\R^p.
       \end{equation}
    Assume that the covariance matrix $\Sigma$ satisfies
    \begin{equation}\label{eq:re_cov_mat}
        \max_{j=1,\dots,p} \Sigma_{jj} \le \frac{1}{2},
        \qquad
        \min_{\vdelta\in \cC_{WRE}(s,c_0):\vdelta\ne \vzero} \frac{\norma{\Sigma^{1/2}\vdelta}_2}{\norma{\vdelta}_2} \ge \kappa.
    \end{equation}
       Let $\design$ be the random matrix in $\R^{n\times p}$ with row vectors $\vx_1,\dots,\vx_n$.
    If
\begin{equation}
n\geq \frac{C L^4 (1+c_0)^2}{\kappa^2} s \log(2ep/s)
    \label{eq:assum-n-p-s-slope-random-design}
\end{equation}
    then, with probability at least $1-3\exp(-C'n/L^4)$ we have
    \begin{equation}\label{eq:main_sub_gauss_mat}
        \max_{j=1,\dots,p} \norm{\design \ve_j}_n^2 \le 1,
        \qquad
        \inf_{\vdelta\in \cC_{WRE}(s,c_0):\vdelta\ne \vzero} \frac{\norm{\design \vdelta}_n}{\norma{\vdelta}_2} \ge \frac{\kappa}{\sqrt{2}}.
    \end{equation}
\end{thm}

\section{Extension to sub-gaussian noise}
\label{sec:extension_to_a_sub_gaussian_noise}

The goal of this section is to show that all results
of the present paper extend to subgaussian noise.
This is due to the following analog of Theorem \ref{thm:main-event}.

\begin{thm}
    \label{thm:main-event-subgaussian}
    Let $\delta_0\in(0,1)$.
    Assume that the components of  $\vxi=(\xi_1,\dots,\xi_n)$
    are independent, {with zero mean}, and subgaussian in the sense that, for some $\sigma>0$,
   \begin{equation}
    \label{eq:assum-subgaussian}
    \E \exp( {\xi_i^2}/{\sigma^2}) \le e,
    \qquad
    i=1,\dots,n.
\end{equation}
  {Let $\design\in \R^{n\times p}$ be any matrix such that $\max_{j=1,\dots,p}\norm{\design\ve_j}_n \le 1$.}  Then, with probability at least $1-\delta_0$
    we have, for all $\vu\in\R^p$,
    \begin{equation*}
    {
                \frac{1}{n} \vxi^T\design \vu
                \le
                40 \sigma
                \max\left(
                    \sum_{j=1}^p u_j^\sharp \sqrt{\frac{\log(2p/j)}{n}}
                    ,
                    \quad
                    \norm{\design\vu}_n
                    \frac{
                        \sqrt{\pi/2} + \sqrt{2\log(1/\delta_0)}
                    }{\sqrt n}
                \right)
                .
    }
    \end{equation*}
\end{thm}

Theorem \ref{thm:main-event-subgaussian} implies that
$
\frac{1}{n} \vxi^T\design \vu
                \le
                \max\left(
                    {\bar H}(\vu)
                    ,
                    {\bar G}(\vu)
                \right)
$
 with probability at least $1-\delta_0$, where ${\bar H}(\cdot)$ and ${\bar G}(\cdot)$ have the same form as ${H}(\cdot)$ and ${G}(\cdot)$ up to numerical constants.Thus, under the sub-gaussian assumption of Theorem \ref{thm:main-event-subgaussian}, all results of Sections \ref{sec:lasso-4} -- \ref{sec:slope} remain valid
up to differences in the numerical constants.

The proof of Theorem \ref{thm:main-event-subgaussian} relies on the following
deviation inequality, which is proved in \Cref{s:proof-main-event-subgaussian}
using symmetrization and contraction arguments.

\begin{proposition}
    \label{prop:sub_gauss_conc}
    Assume that the components of $\vxi$ are independent, {with zero mean},
    and satisfy \eqref{eq:assum-subgaussian}.
    Let $U\subseteq\{\vu\in\R^n: \norma{\vu}_2 \le 1 \}$
    be a subset of the unit ball.
    For any $x>0$,  with probability at least
    $1-\exp(-x)$ we have
    \begin{align}
        \label{eq:main_sub}
        \sup_{\vu\in U} \vxi^T\vu
        \leq
        8 \sigma  \E \left[  \sup_{\vu\in U} \vz^T\vu  \right] +  8 \sigma \sqrt{2x}
        \le
        8 \sigma \Med\left[ \sup_{\vu\in U} \vz^T\vu \right] +  8 \sigma(\sqrt{\pi/2} + \sqrt{2x}),
    \end{align}
    where $\vz$ is a standard normal $\mathcal N(\vzero,I_{n\times n})$ random vector.
\end{proposition}

\appendix

\section{Preliminaries for the proofs}

\begin{lemma}
    \label{lemma:algebra-norm}
    Let $s\in\{1,\dots,p\}$ and $\tau\in[0,1]$.
    For any two $\vbeta,\hbeta\in\R^p$ such that $|\vbeta|_0\le s$  we have
    \begin{equation}\label{eq:algebra-norm}
        \tau \norma{\vu}_*
        +
        \norma{\vbeta}_*
        -
        \norma{\hbeta}_*
        \le
        (1+\tau)\left(\sum_{j=1}^s \lambda_j^2\right)^{1/2} \norma{\vu}_2
        -
        (1-\tau) \sum_{j=s+1}^p \lambda_j u_j^\sharp,
    \end{equation}
    where $\vu = \hbeta - \vbeta = (u_1,\dots,u_p)$
    and $(u_1^\sharp,\dots,u_p^\sharp)$ is a non-increasing rearrangement of
    $(|u_1|,\dots,|u_p|)$.
    If $\lambda_1=\dots=\lambda_p=\lambda$ for some $\lambda>0$,
    then $\norma{\cdot}_* = \lambda \norma{\cdot}_1$
    and \eqref{eq:algebra-norm} yields
    \begin{equation}
        \tau \lambda \norma{\vu}_1
        +
        \lambda \norma{\vbeta}_1
        -
        \lambda \norma{\hbeta}_1
        \le
        (1+\tau)\lambda\sqrt s \norma{\vu}_2
        -
        (1-\tau) \lambda \sum_{j=s+1}^p u_j^\sharp.
    \end{equation}
\end{lemma}
\begin{proof}
    Let $\phi$ be any permutation of $\{1,\dots,p\}$
    such that
    \begin{equation}
        \norma{\vbeta}_*
        = \sum_{j=1}^s \lambda_j | \beta_{\phi(j)} |
        \qquad
        \text{and}
        \qquad
        |u_{\phi(s+1)}|
        \ge
        |u_{\phi(s+2)}|
        \ge
        \dots
        \ge
        |u_{\phi(p)}|.
        \label{eq:permutation-constraint}
    \end{equation}
    By \eqref{eq:maximal-rearrange} applied to $\norma{\hbeta}_*$, we have
    \begin{align*}
        \norma{\vbeta}_*   - \norma{\hbeta}_*
        &\le
        \sum_{j=1}^s \lambda_j \left(|\beta_{\phi(j)}| - |\hat\beta_{\phi(j)}|\right)
        - \sum_{j=s+1}^p \lambda_j |\hat\beta_{\phi(j)}|
        ,\\
        &\le
        \sum_{j=1}^s \lambda_j |u_{\phi(j)}|
        - \sum_{j=s+1}^p \lambda_j |\hat\beta_{\phi(j)}|
        =
        \sum_{j=1}^s \lambda_j |u_{\phi(j)}|
        - \sum_{j=s+1}^p \lambda_j |u_{\phi(j)}|,
    \end{align*}
    since $u_{\phi(j)} = \hat\beta_{\phi(j)} - \beta_{\phi(j)}$ for $j=1,\dots,s$
    and $u_{\phi(j)} = \hat\beta_{\phi(j)}$ for all $j > s$. Since the sequence $\lambda_j$ is non-increasing we have
    $
    \sum_{j=1}^s \lambda_j |u_{\phi(j)}|
    \le
    \sum_{j=1}^s \lambda_j u_j^\sharp
    $.
    Next, the fact that permutation $\phi$ satisfies
    \eqref{eq:permutation-constraint} implies
    $
        \sum_{j=s+1}^p \lambda_j u_j^\sharp
        \le
        \sum_{j=s+1}^p \lambda_j |u_{\phi(j)}|
    $.
    Finally, $\sum_{j=1}^s \lambda_j u_j^\sharp \le (\sum_{j=1}^s\lambda_j^2)^{1/2} \norma{\vu}_2$ by the Cauchy-Schwarz inequality.
\end{proof}

\begin{lemma}
    Let $h:\R^p\rightarrow\R$ be a convex function,
    let $\vf, \vxi\in \R^n$, $\vy=\vf+\vxi$ and let $\design$ be any $n\times p$ matrix.
    If $\hbeta$ is a solution of the minimization problem
    $
        \min_{\vbeta\in\R^p} \left(
        \norm{\design\vbeta - \vy}^2_n + h(\vbeta)
        \right)
    $,
    then $\hbeta$ satisfies for all $\vbeta\in\R^p$
    \begin{equation}
        \norm{\design\hbeta - \vf}^2_n
        -
        \norm{\design\vbeta - \vf}^2_n
        \le
        \frac 2 n \vxi^T\design(\hbeta - \vbeta)
        + h(\vbeta)
        - h(\hbeta)
        - \norm{\design(\hbeta - \vbeta)}^2_n.
        \label{eq:almost-sure}
    \end{equation}
\end{lemma}
\begin{proof}
    Define the functions $f$ and $g$ by the relations $g(\vbeta) = \norm{\design\vbeta - \vy}^2_n$,
    and $f(\vbeta) = g(\vbeta) + h(\vbeta)$ for all $\vbeta\in\R^p$.
    Since $f$ is convex and $\hbeta$ is a minimizer of $f$, it follows that $\vzero$ belongs to the sub-differential of $f$ at $\hbeta$.
    By the Moreau-Rockafellar theorem, there exists $\vv$ in the sub-differential of $h$ at $\hbeta$ such that
    $
        \vzero = \nabla g (\hbeta) + \vv.
        $
    Here, $\nabla g (\hbeta)=\frac2n\design^T(\design \hbeta - \vy)$. Using these remarks and some algebra we obtain
    \begin{align*}
        &\norm{\design\hbeta - \vf}^2_n
        -
        \norm{\design\vbeta - \vf}^2_n
        =
        \frac{2}{n}(\hbeta - \vbeta)^T \design^T(\design \hbeta - \vf) - \norm{\design(\hbeta - \vbeta)}_n^2\\
         & = \frac{2}{n}(\hbeta - \vbeta)^T \design^T(\design \hbeta - \vf) - \norm{\design(\hbeta - \vbeta)}_n^2\
        + (\vbeta - \hbeta)^T(\nabla g(\hbeta) + \vv),\\
        &=
        \frac 2 n \vxi^T\design(\hbeta - \vbeta)
        -
        \norm{\design(\hbeta - \vbeta)}^2_n
        + (\vbeta - \hbeta)^T\vv.
    \end{align*}
    To complete the proof, notice that by definition of the subdifferential of $h$ at $\hbeta$,
    we have
      $  (\vbeta-\hbeta)^T\vv \le h(\vbeta) - h(\hbeta).$
\end{proof}

{\begin{lemma}
 \label{lemma:median-mean}
Let $\vz\sim \mathcal N(\vzero, I_{p\times p})$, and let $f:\R^p\to \R$ be a 1-Lipschitz function, that is, $|f(\vu) - f(\vu')| \le \norma{\vu - \vu'}_2$ for any $\vu,\vu'\in\R^p$. Then,    $|\Med[f(\vz)] - \E[f(\vz)]| \le \sqrt{\pi/2}$.
\end{lemma}

This lemma is proved in the discussion after equation~(1.6) in \cite[page 21]{LT:91}.
 }


\section{Proofs for the Lasso estimator}
\label{sec:proof-lasso}


\begin{proof}[Proof of Theorem \ref{thm:lasso-main-event}]
    Using inequality
    \eqref{eq:almost-sure}  with $h(\cdot) = 2 \lambda \norma{\cdot}_1$ we get that, almost surely, for all $\vbeta\in \R^p$ and all $\vf\in \R^n$,
    \begin{equation}
    \label{eq:first_eq_theo3}
        2\tau \lambda \norma{\hbeta - \vbeta}_1
        +
        \norm{\design\hbeta - \vf}^2_n
        \le
        \norm{\design\vbeta - \vf}^2_n
         - \norm{\design(\hbeta-\vbeta)}_n^2 + \triangle^*
    \end{equation}
    where
    \begin{equation*}
    \triangle^*
    \triangleq
        2 \tau \lambda \norma{\hbeta-\vbeta}_1
        +
        \frac 2 n \vxi^T\design(\hbeta - \vbeta)
        + 2 \lambda\norma{\vbeta}_1
        - 2 \lambda\norma{\hbeta}_1.
    \end{equation*}
    Let $\vu= \hbeta - \vbeta$
    and assume that $\norma{\vbeta}_0\le s$.
    Define
    \begin{equation}
        \tilde H(\vu)
        \triangleq
        \frac{{(4+\sqrt 2)} \sigma}{\sqrt n}
        \left(
            \norma{\vu}_2
            \left(\sum_{j=1}^s\log(2p/j)\right)^{1/2}
            + \sum_{j=s+1}^p u_j^\sharp \sqrt{\log(2p/j)}
        \right).
        \label{eq:def-tilde-H}
    \end{equation}
    Using the Cauchy-Schwarz inequality, it is easy to see that
       \begin{equation}\label{eq:hh}
        H(\vu)
        \le
        \tilde H(\vu) \le \gamma \lambda \left(\sqrt s \norma{\vu}_2 + \sum_{j=s+1}^p u_j^\sharp\right)
            \triangleq F(\vu),
        \qquad
        \forall \vu\in\R^p,
    \end{equation}
     where $H(\cdot)$ is defined in \eqref{eq:def-H-G}, and the last inequality follows from   \eqref{eq:algebra_stirling}
    and \eqref{eq:tuning-lambda}.

    On the event \eqref{eq:main-event},
    using \eqref{eq:hh} and Lemma
    \ref{lemma:algebra-norm} we obtain
    \begin{align*}
        \triangle^*
        &\le 2 \lambda
        \left(
            \tau \norma{\hbeta-\vbeta}_1
            + \norma{\vbeta}_1
            - \norma{\hbeta}_1
        \right)
        + 2 \max(H(\vu), G(\vu)),
        \\
        &\le 2 \lambda
        \left(
            (1+\tau)\sqrt{s}\norma{\vu}_2
            - (1-\tau) \sum_{j=s+1}^p u_j^\sharp
        \right)
        + 2 \max(F(\vu), G(\vu)).
    \end{align*}
    By definition of $\delta(\lambda)$, we have
    $$G(\vu) = \lambda \sqrt s \gamma \sqrt{\log({1}/\delta_0)/(s\log({1}/\delta(\lambda)))} \norm{\design\vu}_n.$$
    We now consider the following two cases.
    \begin{itemize}
        \item[(i)]  {\it Case $G(\vu)>F(\vu)$.}
            Then,
            \begin{equation}
                \norma{\vu}_2
                \le
                \sqrt{\frac{\log({1}/\delta_0)}{s\log({1}/\delta(\lambda))} } \norm{\design\vu}_n.
                \label{eq:case-isometry}
            \end{equation}
            Thus
            \begin{align}
                \triangle^*
                &\le
                2 \lambda
                (1+\tau)\sqrt s \norma{\vu}_2
                + 2 G(\vu),
                \nonumber
                \\
                &\le
                2 \lambda \sqrt s (1+\tau+\gamma) \sqrt{\frac{\log({1}/\delta_0)}{s\log({1}/\delta(\lambda) )}} \norm{\design\vu}_n,
                \nonumber
                \\
                &\le
                \lambda^2 s (1+\tau+\gamma)^2 \frac{\log({1}/\delta_0)}{s\log({1}/\delta(\lambda))} + \norm{\design\vu}_n^2.
                \label{eq:case-prediction}
            \end{align}

        \item[(ii)]{\it
            Case $G(\vu)\le F(\vu)$.}
            In this case, we get
            \begin{align*}
                    \triangle^*
                    &\le
                    2 \lambda
                    \left(
                        (1 + \gamma + \tau)\sqrt s \norma{\vu}_2
                        -
                        (1 - \gamma - \tau)\sum_{j=s+1}^p u_j^\sharp
                    \right)
                    \triangleq
                    \triangle
                    .
            \end{align*}
            If $\triangle > 0$, then  $\vu$ belongs to the cone $\cC_{SRE}(s,c_0)$ and
             we can use the $SRE(s,c_0)$ condition, which yields $|\vu|_2\le \norm{\design \vu}_n/\theta(s,c_0)$. Therefore,
            \begin{eqnarray}
                \triangle^*
                \le
                \triangle
                &\le&
                    \frac{2(1+\gamma +\tau)\lambda \sqrt s}{\theta(s,c_0)}
                \norm{\design \vu }_n
                \le \left(
                    \frac{(1+ \gamma +\tau)\lambda \sqrt s}{\theta(s,c_0)}
                \right)^2
                + \norm{\design \vu}_n^2.
                \label{eq:case2-prediction}
            \end{eqnarray}
            If $\triangle \le 0$, then \eqref{eq:case2-prediction} holds trivially.
       \end{itemize}
       Combining \eqref{eq:case-prediction} and \eqref{eq:case2-prediction} with \eqref{eq:first_eq_theo3} completes the proof of  \eqref{eq:soi-lasso1}.

    Let now $\vf = \design\vbeta^*$ for some $\vbeta^*\in\R^p$ with $\norma{\vbeta^*}_0 \le s$.
Then, \eqref{eq:soi-lasso1} with $\vbeta=\vbeta^*$ implies
\begin{equation}
 \label{eq:estimation-ell1-lasso}
        2\tau \norma{\hbeta - \vbeta^*}_1
        \le
        C_{\gamma,\tau}(s,\lambda,\delta_0)\lambda s.
\end{equation}
Next, we show that
\begin{equation}
\label{eq:estimation-ell2-lasso}
        \norma{\hbeta - \vbeta^*}_2
        \le
        \frac{C_{\gamma,0}(s,\lambda,\delta_0)}{1+ \gamma} \frac{\lambda \sqrt s}{\theta^2(s,\frac{1+\gamma}{1-\gamma})}  .
\end{equation}
To prove
    \eqref{eq:estimation-ell2-lasso},
    we take $\tau=0$,  and consider the cases (i) and (ii) as above with $\vu=\hbeta-\vbeta^*$.
    \begin{itemize}
        \item If $G(\vu)>F(\vu)$, then from \eqref{eq:first_eq_theo3} and \eqref{eq:case-prediction} with $\tau=0$ and $\vbeta=\vbeta^*$ we get 
            \begin{equation*}
                                \norm{\design(\hbeta-\vbeta^*)}_n^2
               \le \lambda^2(1+\gamma)^2
               \frac{\log({1}/\delta_0)}{\log({1}/\delta(\lambda))}
.
            \end{equation*}
            This and \eqref{eq:case-isometry} imply
            \begin{equation}\label{eq:gf}
                \norma{\hbeta - \vbeta^*}_2
                \le
                                (1+\gamma)\lambda\sqrt s
                \left(
                    \frac{\log(1/\delta_0)}{s\log(1/\delta(\lambda))}
                \right).
            \end{equation}
        \item If $G(\vu)\le F(\vu)$, then it follows from \eqref{eq:first_eq_theo3} with $\vbeta=\vbeta^*$ that
    $ \triangle^*\ge 0$ almost surely, and thus
            $\triangle\ge \triangle^* \ge 0$. Hence,
            $\vu=\hbeta - \vbeta^*\in\cC_{SRE}(s,\frac{1+\gamma}{1-\gamma})$.
            Thus, we can apply the $SRE(s,\frac{1+\gamma}{1-\gamma})$
            condition, which yields 
            \begin{equation}\label{eq:fg}
                \norma{\hbeta - \vbeta^*}_2
                \le
                \frac{\norm{\design\vu}_n}{\theta(s,\frac{1+ \gamma}{1-\gamma})}
                \le
                \frac{(1+ \gamma) \lambda \sqrt s}{\theta^2(s,\frac{1+ \gamma}{1-\gamma})}
            \end{equation}
            where the second inequality is due to the combination of \eqref{eq:first_eq_theo3} and \eqref{eq:case2-prediction} with $\vbeta=\vbeta^*$, $\tau=0$.
     \end{itemize}
     Putting together \eqref{eq:gf} and \eqref{eq:fg} proves \eqref{eq:estimation-ell2-lasso}.
           To conclude, it is enough to notice that $\theta^2(s,\frac{1+ \gamma}{1-\gamma})\geq \theta^2(s,c_0)$ and then to bound $\norma{\hbeta - \vbeta^*}_q$ from above using \eqref{eq:estimation-ell1-lasso}, \eqref{eq:estimation-ell2-lasso} and the norm
     interpolation inequality
     $\norma{\hbeta - \vbeta^*}_q\le \norma{\hbeta - \vbeta^*}_1^{2/q-1}\norma{\hbeta - \vbeta^*}_2^{2-2/q}.$
\end{proof}

\begin{proof}[Proof of Corollary \ref{cor:lasso-E}]
    Let $\gamma=1/2$, $\tau=1/4$, and let $\delta_0^*$ be defined in \eqref{eq:choice-explicit-delta0}. Set
        \begin{align*}
        Z\triangleq
        \frac{16\log(2ep/s)}{49\lambda^2 }
        \sup_{\vbeta\in\R^p:\norma{\vbeta}_0\le s}
        \left(
            \frac \lambda 2 \norma{\hbeta - \vbeta}_1
            +
            \norm{\design\hbeta - \vf}^2_n
            -
            \norm{\design\vbeta - \vf}^2_n
        \right).
           \end{align*}
    For all $\delta_0\in(0,\delta_0^*]$,
    by Theorem \ref{thm:lasso-main-event} we have $Z\le
        \log(1/\delta_0)$
    with probability at least $1-{\delta_0/2}$.
    That is, $\mathbb P(Z> t) \le {\frac{e^{-t}}{2}}$
    for all $t\ge T\triangleq \log(1/\delta_0^*)=\frac{s \log(2ep/s)}{\theta^2(s,7)}$.
    By integration,
    \begin{equation*}
        \E[Z]
        \le
        \int_0^{\infty}\mathbb P(Z > t) dt
        \le T + {\int_{T}^{\infty} \frac{e^{-t}}{2}} dt
        \le T + {\frac 1 2},
    \end{equation*}
    which yields
    \begin{multline*}
        \E \sup_{\vbeta\in\R^p:\norma{\vbeta}_0\le s}
        \left(
            \frac \lambda 2 \norma{\hbeta - \vbeta}_1
            +
            \norm{\design\hbeta - \vf}^2_n
            -
            \norm{\design\vbeta - \vf}^2_n
        \right)
        \\
        \le \frac{49}{16} \left( \frac{\lambda^2 s}{\theta^2(s,7)}
        + \frac{ \lambda^2 s}{ 2 s\log(2ep/s)}\right)
        \le
        \frac{49}{16}\left( \frac{\lambda^2 s}{\theta^2(s,7)}
        + \frac{\lambda^2 s}{ {2} \log(2ep)}\right).
    \end{multline*}
    This completes the proof of \eqref{eq:soi-expectation-lasso}.

    Let now $\vf = \design\vbeta^*$ for some $\vbeta^*\in\R^p$ with $\norma{\vbeta^*}_0 \le s$.
    For all $t\ge T = \log(2/\delta_0^*)$, by Theorem \ref{thm:lasso-main-event} we have
    \begin{equation*}
        Z_q \triangleq
        \frac{8s\log(2ep/s)}{49\lambda s^{1/q}}
        \norma{\hbeta-\vbeta^*}_q
        \le
        t
    \end{equation*}
    with probability at least $1-\frac{e^{-t}}{2}$.
    To prove \eqref{eq:esimtation-ell2-expectation-lasso}, it remains to note that
    \begin{equation*}
        \E[Z_q^q]
        = \int_0^{\infty} qt^{q-1} \; \mathbb P(Z_q>t) dt
        \le
        \int_0^T qt^{q-1} dt
        + \int_T^{\infty} \frac{qt^{q-1} e^{-t}}{{2}} dt
        \le
        T^q
        {+ \frac{q}{2} \Gamma(q)}
        \le T^q + {1}.
    \end{equation*}
\end{proof}


\section{Lasso with adaptive choice of $\lambda$}
\label{sec:proof-lepski}

\begin{proof}[Proof of \Cref{th:lepski}]
    In this proof, we set for brevity $\Pro = \Pro_{\vbeta^*}$.
    Fix $s\in [1,s_*]$ and assume that $\vbeta^*\in B_0(s)$. If $s<b_M$, let $m_0$ be the index such that $b_{m_0}$ is the minimal element greater than $s$ in the collection $\{{b_m}, m=2,\dots, M\}$. Then, $b_{m_0-1}\le s < b_{m_0}$. If $s\in [b_M, s_*]$, set $m_0=M$.
    For any $a>0$ we have
    \begin{align}\label{th:lepski:eq2}
        \Pro \left( d(\boldsymbol{\tilde\beta} , \vbeta^*)
        \ge a \right) \le
        \Pro \left(d(\boldsymbol{\tilde\beta},  \vbeta^*)
        \ge a, \hat m \le m_0 \right) +  \Pro \left(\hat m \ge m_0+1 \right).
    \end{align}
    On the event $\{\hat m \le m_0\}$ we have
    \begin{align}\label{th:lepski:eq3}
        d(\hbeta_{b_{\hat m}} , \hbeta_{b_{m_0}})
        & \le
        \sum_{k=\hat m+1}^{m_0}
        d(\hbeta_{b_{k}} , \hbeta_{b_{k-1}})
        \le
        2C_0\sigma\sum_{k=\hat m+1}^{m_0}   \sqrt{\frac{b_k\log({2}ep/b_k)}{n}}
        \\
        & \le
        2C_0\sigma\sum_{k=2}^{m_0}   \sqrt{\frac{b_k\log({2}ep/b_k)}{n}}
        \le
        c'C_0\sigma   \sqrt{\frac{b_{m_0}\log({2}ep/b_{m_0})}{n}}
    \end{align}
    where $c'>2$ is an absolute constant. We deduce that, on the event $\{\hat m \le m_0\}$,
    \begin{eqnarray}\label{th:lepski:eq4}
        d(\boldsymbol{\tilde\beta} , \vbeta^*) &\le& d(\hbeta_{b_{\hat m}}, \hbeta_{b_{m_0}}) +
        d(\hbeta_{m_0}, \vbeta^*)
        \le
        c'w(b_{m_0})  + d(\hbeta_{m_0}, \vbeta^*)\,.
    \end{eqnarray}
    The following two cases are possible.
    \begin{itemize}
        \item[(i)] {\it Case} $s<b_M$. Then, by definition of $m_0$ we have $b_{m_0}/2=b_{m_0-1}\le s < b_{m_0}$. Since the function $w(\cdot)$ is increasing on $[1,p]$, we easily deduce that
            $$
            w(b_{m_0})/\sqrt{2} \le w(s)\le w(b_{m_0}).
            $$
        \item[(ii)] {\it Case} $s\in [b_M, s_*]$. Then, $b_{m_0}=b_M \le s$, while $s\le 2b_M$. Therefore, in this case $b_{m_0}\le s < 2b_{m_0}$, which implies
            $$
            w(b_{m_0}) \le w(s)\le \sqrt{2} \,w(b_{m_0}).
            $$
    \end{itemize}
    In both cases  (i) and (ii), we have
    $
    w(s)\ge w(b_{m_0})/\sqrt{2}
    .$ This remark and the fact that \eqref{th:lepski:eq4} holds on the event $\{\hat m \le m_0\}$ imply
    \begin{eqnarray}\nonumber
        \Pro \left(d(\boldsymbol{\tilde\beta},\vbeta^*) \ge {\sqrt{2}}({\sqrt{2}}+c') w(s) , \ \hat m \le m_0\right)&\le&
        \Pro \left(d(\boldsymbol{\tilde\beta},\vbeta^*) \ge ({\sqrt{2}}+c') w(b_{m_0}) , \ \hat m \le m_0\right) \\
        &\le&\Pro \left(d(\hbeta_{m_0},\vbeta^*)\ge  {\sqrt{2}}\,w(b_{m_0})        \right) .   \label{from_corollary0}
    \end{eqnarray}
    Next, in both cases (i) and (ii), we have $s\le 2b_{m_0}$, which implies that $\vbeta^* \in B_0(2b_{m_0})$.
    Using this fact together with \eqref{from_corollary0} and  \eqref{from_corollary_times2} we obtain
    \begin{eqnarray}\nonumber
        &&\sup_{\vbeta^*\in B_0(s)}\Pro \left(d(\boldsymbol{\tilde\beta},\vbeta^*) \ge {\sqrt{2}}({\sqrt{2}}+c') w(s) , \ \hat m \le m_0\right)
        \\
        && \quad
        \le
        \sup_{\vbeta^*\in B_0(2b_{m_0})} \Pro \left(d(\hbeta_{m_0},\vbeta^*)\ge  {\sqrt{2}}\,w(b_{m_0})\right)
        \le \left(\frac {2b_{m_0}}{ p}\right)^{2b_{m_0}} \le \left(\frac {2s}{ p}\right)^{2s}\label{from_corollary1}
    \end{eqnarray}
    where we have used that the function $b\mapsto \left(b/ p\right)^{b}$ is decreasing on the interval $[1,p/e]$ and,  in both cases (i) and (ii), $2b_{m_0}\le 2s_*\le p/e$.

    We now estimate the probability $\Pro \left(\hat m \ge m_0+1 \right)$. We have
    \begin{align*}
        \Pro \left(\hat m \ge m_0+1 \right)& \le \sum_{m=m_0+1}^M \Pro \left(\hat m = m\right).
    \end{align*}
    Now, from the definition of $\hat m$ we obtain
    \begin{align*}
        \Pro \left(\hat m = m\right)&\le \sum_{k=m}^M \Pro\left(d(\hbeta_{b_k},\hbeta_{b_{k-1}}) > 2w(b_k)\right) \\
                                    &\le
        \sum_{k=m}^M \Pro\left(d(\hbeta_{b_k},\vbeta^*) > w(b_k)\right)
        +
        \sum_{k=m}^M \Pro\left(d(\hbeta_{b_{k-1}}, \vbeta^*) > w(b_k)\right)
        \\
        &\le
        2\sum_{k=m-1}^M \Pro\left(d(\hbeta_{b_k},\vbeta^*) > w(b_k)\right)\triangleq 2\sum_{k=m-1}^M a_k,
    \end{align*}
    where we have used that $b_k>b_{k-1}$ and the monotonicity of $w(\cdot)$. Thus,
    \begin{equation}\label{doublsum}
        \Pro \left(\hat m \ge m_0+1 \right) \le  2\sum_{m=m_0+1}^M\sum_{k=m-1}^M a_k
        %
    \end{equation}
    The double sum in \eqref{doublsum} is non-zero only if $m_0<M$. This implies that $s<b_{m_0}$, and hence $\vbeta^* \in B_0(b_{m_0})\subset B_0(b_{k})$ for all $k\ge m_0$. Therefore, using \eqref{from_corollary} we obtain
    \begin{equation}
        \sup_{\vbeta^*\in B_0(s)} \Pro \left(\hat m \ge m_0+1 \right) \le 2\sum_{m=m_0+1}^M\sum_{k=m-1}^M \sup_{\vbeta^*\in B_0(b_{k})} a_k \le 2 M^2
        \max_{k\ge m_0} \left(\frac {b_{k}}{ p}\right)^{b_{k}}.
    \end{equation}
    Recall that
    $M\le \log_2(2p)$. Note also that the function $b\mapsto \left(\frac {b}{ p}\right)^{b}$ is decreasing on the interval $[1,p/e]$, while $b_j\le s_*$  for all $j$ and $s_*\le p/e$ by assumption. Finally, in both cases (i) and (ii), $b_{m_0}\le 2s$. Using these remarks we get
    \begin{align}\label{kk}
        \sup_{\vbeta^*\in B_0(s)}  \Pro \left(\hat m \ge m_0+1 \right)& \le  2 M^2 \left(\frac {b_{m_0}}{ p}\right)^{b_{m_0}}
        \le 2 (\log_2(p))^2 \left(\frac {2s}{ p}\right)^{2s}.
    \end{align}
    Combining this bound with \eqref{from_corollary1} and \eqref{th:lepski:eq2} where we set $a={\sqrt{2}}({\sqrt{2}}+c') w(s)$ proves \eqref{th:lepski:eq1}. Finally, inequality \eqref{th:lepski:s} follows from \eqref{kk} and the relations
    $\{{\hat s}\le s\}= \{b_{\hat m-1}\le s\}=\{b_{\hat m}/2\le s\}\supseteq \{b_{\hat m}\le b_{m_0}\} =\{\hat m \le m_0\}$.
\end{proof}
\vspace{-8mm}
{
\begin{proof}[Proof of \Cref{th:lepski-bis}]
The constant $C_0$ in \Cref{th:lepski-bis} is chosen to satisfy $C_0=\sqrt{2} C_0'$. Thus,  inequalities \eqref{from_corollary_times3}  and  \eqref{from_corollary_times4} imply the inequalities of the same form as \eqref{from_corollary}  and  \eqref{from_corollary_times2}.
Note that $w(b) = C_0\sigma b^{1/q}\sqrt{\frac{\log(2ep/b)}{n}}$ is increasing in $b$ for $b\in[1,p]$. Note also that \eqref{th:lepski:eq3} remains valid if we replace there the expressions of the form $\sqrt{b \log(2ep/b)}$ by $ b^{1/q}\sqrt{\log(2ep/b)}$. Using these remarks, we obtain the result of \Cref{th:lepski-bis} by repeating, with minor modifications, the proof of \Cref{th:lepski} if we set $
    d(\vbeta, \vbeta')= \norma{\vbeta - \vbeta'}_q, \ \forall \vbeta, \vbeta'\in \R^p,
    $
 and replace the references \eqref{from_corollary}  and  \eqref{from_corollary_times2} with \eqref{from_corollary_times3}  and  \eqref{from_corollary_times4}, respectively. We omit further details.
    \end{proof}
}


\section{Proofs for the Slope estimator}
\label{sec:proof-slope}

\begin{proof}[Proof of Theorem \ref{th:slope}] 
    By
    \eqref{eq:almost-sure} with $h(\cdot) = 2 \norma{\cdot}_*$, we have that
    almost surely, for all $\vbeta\in \R^p$ and all $\vf\in \R^n$,
    \begin{equation}\label{D1}
        2\tau \norma{\hbeta - \vbeta}_*
        +
        \norm{\design\hbeta - \vf}^2_n
        \le
        \norm{\design\vbeta - \vf}^2_n
         - \norm{\design(\hbeta-\vbeta)}_n^2+ \triangle^*
    \end{equation}
    where $\triangle^*\triangleq
        2 \tau \norma{\hbeta-\vbeta}_*
        +
        \frac 2 n \vxi^T\design(\hbeta - \vbeta)
        + 2 \norma{\vbeta}_*
        - 2 \norma{\hbeta}_*$.
    Let $\vu = \hbeta - \vbeta$.
   Now consider the event \eqref{eq:main-event}
    and the quantities $H(\vu)$ and $G(\vu)$ defined in
    \eqref{eq:def-H-G}.
    On the event \eqref{eq:main-event},
    using \eqref{eq:hh} and
    Lemma \ref{lemma:algebra-norm} we obtain
    \begin{align*}
        \triangle^*
        &\le 2
        \left(
            \tau \norma{\hbeta-\vbeta}_*
            + \norma{\vbeta}_*
            - \norma{\hbeta}_*
        \right)
        + 2 \max(H(\vu), G(\vu)),
        \\
        &\le 2
        \left(
            (1+\tau)\norma{\vu}_2 \Lambda(s)
            - (1-\tau) \sum_{j=s+1}^p \lambda_j u_j^\sharp
        \right)
        + 2 \max(\tilde H(\vu), G(\vu)),
    \end{align*}
    where $\tilde H(\cdot)$ is defined in \eqref{eq:def-tilde-H}
    and $\Lambda(s) = (\sum_{j=1}^s \lambda_j^2)^{1/2}$.
    We now consider the following two cases.
    \begin{itemize}
        \item[(i)]  {\it Case $\tilde H(\vu)\le G(\vu)$.}
            In this case, the definition of $\tilde H$ (cf. \eqref{eq:def-tilde-H}),
            implies
            \begin{equation}
                {
                \norma{\vu}_2
                \le \frac{ G(\vu)}{ (4+\sqrt 2)(\sum_{j=1}^s\log(2p/j))^{1/2} }
                \le \norm{\design\vu}_n \sqrt{\frac{\log(1/\delta_0)}{s\log(2p/s)}}
                .
                }
                \label{eq:bound-on-G-slope-case1}
            \end{equation}
            where we used \eqref{eq:algebra_stirling} for the second inequality.
            As the weights \eqref{eq:recommendation-weights-slope} and the constant
            $A$ satisfies \eqref{eq:condition-A}, by \eqref{eq:algebra_stirling} we also have
            $G(\vu) \le \gamma \Lambda(s) \sqrt{\log({1}/\delta_0)/(s\log(2p/s))}\norm{\design\vu}_n$.
            Using \eqref{eq:bound-on-G-slope-case1} we obtain
            \begin{align}
                \triangle^*
                &\le
                2 (1+\tau)\Lambda(s) \norma{\vu}_2
                + 2 G(\vu)
                \nonumber
                \\
                &\le
                2(1+\gamma+\tau) \Lambda(s) \sqrt{\log({1}/\delta_0)/(s\log(2p/s))} \norm{\design\vu}_n,
                \nonumber
                \\
                &\le
                (1+\gamma+\tau)^2 \Lambda^2(s) \left(\frac{\log({1}/\delta_0)}{s\log(2p/s)}\right)
                + \norm{\design\vu}_n^2.
                \label{eq:bound-prediction-trianglestar-slope-case2}
            \end{align}
        \item[(ii)]  {\it Case $\tilde H(\vu)>G(\vu)$.}
            In this case,
            \begin{equation*}
                \triangle^*
                \le
                2 (1+\gamma+\tau)\norma{\vu}_2 \Lambda(s)
                -
                2 (1-\gamma-\tau)\sum_{j=s+1}^p \lambda_j u_j^\sharp
                \triangleq\triangle
                .
            \end{equation*}
            If $\triangle\le 0$ then \eqref{eq:soi-slope} holds trivially in view of \eqref{D1}.
            If $\triangle>0$, then $\vu$ belongs to the
            cone $\cC_{WRE}(s,c_0)$, and we can use
            the $WRE(s,c_0)$ condition, which yields
            \begin{align}\label{eq:ccc}
                 \triangle^* \le \triangle
                \le
                2 (1+\gamma+\tau)
                \Lambda(s)
                \norma{\vu}_2
                &\le
                \frac{2 (1+\gamma+\tau) \Lambda(s) \norm{\design\vu}_n }{\vartheta(s,c_0)}\nonumber
                 \\
                &\le
                \frac{(1+\gamma+\tau)^2 \Lambda^2(s)}{\vartheta^2(s,c_0)}
                + \norm{\design\vu}_n^2.
            \end{align}
    \end{itemize}
    Combining the last inequality with \eqref{eq:bound-prediction-trianglestar-slope-case2} and \eqref{D1} completes the proof of  \eqref{eq:soi-slope}.

    Let now $\vf = \design\vbeta^*$ for some $\vbeta^*\in\R^p$ with $\norma{\vbeta^*}_0 \le s$.
    Then, \eqref{D1} implies that
    $\triangle^*\ge 0$ almost surely.
    The estimation bound \eqref{eq:estimation-well-specified-slope} is
    a direct consequence of \eqref{eq:soi-slope}.
    To prove \eqref{eq:esimtation-ell2-slope},
    we set in what follows $\vu=\hbeta-\vbeta^*$, $\tau=0$ so that $c_0 = \frac{1+\gamma}{1-\gamma}$ and consider
    the same two cases as above.
    \begin{itemize}
        \item If $\tilde H(\vu)\le G(\vu)$, then \eqref{eq:bound-on-G-slope-case1} holds.
            Combining this bound 
            with \eqref{D1} -- \eqref{eq:bound-prediction-trianglestar-slope-case2} we conclude that \eqref{eq:esimtation-ell2-slope} is satisfied in this case.
        \item If $\tilde H(\vu)>G(\vu)$,
            then $\triangle\ge \triangle^*$. It follows from \eqref{D1} with $\vbeta=\vbeta^*$ that $\triangle^*\ge 2\norm{\design\vu}_n^2$. Thus, $\triangle\ge 0$.
            Therefore, $\vu$ belongs to the cone $\cC_{WRE}(s,\frac{1+\gamma}{1-\gamma})$, and we can apply
            the $WRE(s,\frac{1+\gamma}{1-\gamma})$ condition, which yields
            $\norma{\vu}_2 \le\norm{\design\vu}_n/\vartheta(s,\frac{1+\gamma}{1-\gamma})$.
            Combining this bound with the inequality $2\norm{\design\vu}_n^2 \le\triangle $ and \eqref{eq:ccc}
            we obtain that \eqref{eq:esimtation-ell2-slope} is satisfied.
    \end{itemize}
\end{proof}


\section{Bound on the stochastic error}
\label{s:proof-main-event}

Here, we prove Theorem \ref{thm:main-event}. The proof is based on a sequence of propositions.

\begin{proposition}
    \label{prop:event-lasso}
    Let $g_1,\dots,g_p$ be zero-mean Gaussian random variables
    with variance at most $\sigma^2$.
    Denote by $(g_1^\sharp,\dots,g_p^\sharp)$ be a non-increasing
    rearrangement of $(|g_1|,\dots,|g_p|)$.
    Then
    \begin{equation*}
        \mathbb P
        \left(
            \frac 1 {s\sigma^2} \sum_{j=1}^s (g_j^\sharp)^2
            >
            t \log\Big(\frac{2p}{s}\Big)
        \right) \le
        \left(\frac{2p}{s}\right)^{1  - \frac{3t} 8}
    \end{equation*}
    for all $t>0$ and $ s\in\{1,\dots,p\}$.
\end{proposition}

\begin{proof}
  By Jensen's inequality we have
    \begin{equation}\label{expmom}
        \E \exp\left(
            \frac 3 {8s\sigma^2} \sum_{j=1}^s (g_j^\sharp)^2
        \right)
        \le
        \frac 1 s \sum_{j=1}^s \E \exp\left( \frac{3(g_j^\sharp)^2}{8\sigma^2} \right)
        \le
        \frac 1 s \sum_{j=1}^p \E \exp\left( \frac{3g_j^2}{8\sigma^2}\right)
        \le \frac {2p} s
    \end{equation}
    where we have used
    the fact that $\E[\exp( 3\eta^2/8)] = 2$ when $\eta\sim\cN(0, 1)$.
    The Chernoff bound completes the proof.
\end{proof}

\begin{proposition}
    \label{prop:event-slope} Under the assumptions of Proposition \ref{prop:event-lasso},
    \begin{equation}
        \Pro\left(
            \max_{j=1,\dots,p}
            \;
            \frac{g_j^\sharp}{\sigma \sqrt{\log(2p/j)}}
            \le
            4
        \right) \ge \frac 1 2.
    \end{equation}
\end{proposition}

\begin{proof}
    Proposition
     \ref{prop:event-lasso} with $t = 16 / 3$,
    and the inequality
    $(g_j^\sharp)^2 \le \frac 1 j \sum_{k=1}^j (g_k^\sharp)^2$
    imply
    \begin{equation}\label{eq:aa}
        \Pro\left(
            (g_j^\sharp)^2 \le \frac{16 \sigma^2}{3} \log(2p/j)
        \right)
        \ge 1 - \frac{j}{2p}, \quad j=1,\dots,p.
    \end{equation}
    Let $q\ge 0$ be the integer such that $2^q\le p < 2^{q+1}$. Applying \eqref{eq:aa} to $j=2^l$ for $l=0,\dots,q-1,$ and using
    the union bound, we obtain that the event
    \begin{equation*}
        \Omega_0 \triangleq
        \left\{
            \max_{l=0,\dots,q-1}
                          \frac{g_{2^l}^\sharp \sqrt 3}{4 \sigma \sqrt{\log(2p/2^l)}}
                       \le 1
        \right\}
    \end{equation*}
    satisfies
    $\Pro(\Omega_0)
        \ge 1 - \sum_{l=0}^{q-1} \frac{2^l}{2p}
        = 1 - \frac{2^{q}-1}{2p}
        \ge  1/2.
    $
     For any $j < 2^{q}$,
    there exists $l\in\{0,\dots,q-1\}$ such that $2^l\le j <2^{l+1}$
    and thus, on the event $\Omega_0$,
    \begin{equation*}
        g_j^\sharp
        \le g_{2^l}^\sharp
        \le
        \frac{4 \sigma}{\sqrt 3} \sqrt{\log(2p/2^l)}
        \le
        \frac{4 \sigma}{\sqrt 3} \sqrt{\log(4p/j)}
        \le
        4 \sigma \sqrt{\log(2p/j)}, \quad \forall \ j < 2^{q}.
    \end{equation*}
    Next, for $2^{q}\le j\le p$ we have
    \begin{equation*}
        g_j^\sharp
        \le g_{2^{q-1}}^\sharp
        \le
        \frac{4 \sigma}{\sqrt 3} \sqrt{\log(2p/2^{q-1})}
        <
        \frac{4 \sigma}{\sqrt 3} \sqrt{\log(8p/j)}
        \le
        4 \sigma \sqrt{\log(2p/j)}.
    \end{equation*}
    Thus, on the event $\Omega_0$ we have
    $
        g_j^\sharp
        \le
        4 \sigma \sqrt{\log(2p/j)}
   $
    for all $j=1,\dots,p$.
\end{proof}


A function $N:\R^p\rightarrow [0,\infty)$ will be called {positive} homogeneous if
$N(a\vu) = a N(\vu)$ for all $a\geq0, \vu\in\R^p$
    and $N(\vu)>0$ for $\vu\ne \vzero$.

{
\begin{proposition}
    \label{prop:deviation-abstract-norm-N}
    Let $\delta_0\in(0,1)$.
    Assume that $\vxi\sim \mathcal N(\vzero,\sigma^2 I_{n\times n})$.
    Let $N:\R^p\rightarrow [0,+\infty)$ be a {positive} homogeneous function.    Assume that the event
    \begin{equation*}
        \Omega_4 \triangleq \left\{
                \sup_{\vv\in\R^p: N(\vv) {\le} 1} \frac 1 n  \vxi^T\design\vv
                \le 4
        \right\}
    \end{equation*}
    satisfies $\mathbb P(\Omega_4) \ge 1/2$.
    Then for all $\delta_0\in(0,1)$ we have
    \begin{equation*}
        \mathbb P\left(
            \forall \vu\in\R^p:
            \quad
        \frac 1 n \vxi^T\design\vu
        \le (4+\sqrt 2) \max\left(N(\vu), \norm{\design\vu}_n \sigma \sqrt{\frac{\log(1/\delta_0)}{n}} \right)
        \right)
        \ge 1 - \delta_0/2.
        \label{eq:new-deviation-abstract-norm-N}
    \end{equation*}
\end{proposition}
\begin{proof}
    By homogeneity, it is enough to consider only $\vu\in\R^p$ such that
    $\max\left(N(\vu), \norm{\design\vu}_n / L \right) = 1$
    \text{where}
    $L \triangleq (n/ (\sigma^2\log(1/\delta_0)))^{1/2}$.
    Define $T\subset\R^p$ and $f:\R^n\rightarrow \R$ by
    \begin{equation}
        T \triangleq \left\{\vu\in\R^p: \max\left(N(\vu), \frac 1 L \norm{\design\vu}_n\right)   \le 1\right\},
        \qquad
        f(\vv) \triangleq
        \sup_{\vu\in T}
        \frac 1 n (\sigma \vv)^T\design\vu
        \label{eq:def-T-f}
    \end{equation}
    for all $\vv\in\R^n$.
  Then, $f$ is a Lipschitz function with Lipschitz constant $\sigma L/\sqrt n$.
    Thus, by \cite[Inequality (1.4)]{LT:91}, we have
    with probability at least $1-\delta_0/2$,
    \begin{align*}
        \sup_{\vu\in T}
        \frac 1 n \vxi^T\design\vu
        &\le
        \Med\left[
                \sup_{\vu\in T}
                \frac 1 n \vxi^T\design\vu
        \right]
        + \sigma L \sqrt{\frac{2\log(1/\delta_0)}{n}}
        \\
       &
       { \le
        \Med\left[
                \sup_{\vu\in\R^p: N(\vv) {\le} 1}
                \frac 1 n \vxi^T\design\vu
        \right]
        + \sigma L \sqrt{\frac{2\log(1/\delta_0)}{n}}
        }
        \\
        &\le
        4
        + \sigma L \sqrt{\frac{2\log(1/\delta_0)}{n}} = 4+\sqrt 2,
    \end{align*}
    where we used the fact that $\mathbb P(\Omega_4)\ge 1/2$ to bound from above the median.
\end{proof}

\begin{proof}[Proof of Theorem \ref{thm:main-event}] Set
    \begin{equation}\label{eq:nu}
        N(\vu) = \sum_{j=1}^p u_j^\sharp \sigma \sqrt{\frac{\log(2p/j)}{n}}.
    \end{equation}
    Then, using that $g_j = \frac{1}{\sqrt n} \vxi^T\design\ve_j$ for all $j=1,...,p$ we have
    \begin{equation}\label{eq:nuu}
                \sup_{\vu\in\R^p: N(\vu) \le 1} \frac 1 n \vxi^T\design\vu
                \le
                \sup_{\vu\in\R^p: N(\vu) \le 1}
                \sum_{j=1}^p u_j^\sharp \sigma \sqrt{\frac{\log(2p/j)}{n}} \frac{g_j^\sharp}{\sigma\sqrt{\log(2p/j)}}
                \le
                \max_{j=1,\dots,p} \frac{g_j^\sharp}{\sigma\sqrt{\log(2p/j)}}.
    \end{equation}
    By Proposition \ref{prop:event-slope}, we have $\mathbb P(\Omega_4) \ge 1/2$, where $\Omega_4$ is the event introduced in  Proposition~\ref{prop:deviation-abstract-norm-N}. {Therefore, Theorem \ref{thm:main-event} follows immediately from Proposition \ref{prop:deviation-abstract-norm-N}.}
\end{proof}
}


\section{Tools for lower bounds}

\begin{lemma}\label{lem:verzelen}
    For any integers $p\ge 2$, $n\ge 1$, $s\in [1,p/2]$, and any matrix $\design\in\R^{n\times p}$, there exists a subset $\Omega$ of the set $\big\{1, 0,-1\big\}^p$
    with the following properties:
    \begin{eqnarray*}
        &&\norma{\vomega}_0=s, \quad \text{and} \quad     \norm{\design\vomega}^2_n \le \bar\theta_{\max}^2(\design, 1) s, \quad \forall  \vomega\in \Omega,\\ \vspace{2mm}
        && \log (|\Omega|) \ge  \tilde c s\log\left(\frac{ep}{s}\right)
    \end{eqnarray*}
    where $\tilde c>0$ is an absolute constant
    , and
    \begin{eqnarray*}
        \norma{\vomega - \vomega'}_q \ge (s/4)^{1/q}, \quad \forall \ 1\le q\le\infty,
    \end{eqnarray*}
    for any two distinct elements $\vomega$ and $\vomega'$of $\Omega$.
\end{lemma}

The proof of this lemma is omitted since it closely follows the argument in \cite[p.79--80]{verzelen2012}.

\section{Random design matrices}

\begin{proof}[Proof of Theorem \ref{thm:design-random-matrices-slope}]
    In this proof, we denote by
    $C_i$ absolute positive constants.
    Using \eqref{C},    the second inequality in \eqref{eq:re_cov_mat}, and \eqref{eq:algebra_stirling} we get the inclusion
    $
    \cC_{WRE}(s,c_0)\cap
    \{\vv\in\R^p:\norma{\Sigma^{1/2}\vv}_2=1\}
    \subset T
    $ where
    \begin{equation}
        \label{T}
        T \triangleq \left\{\vv\in\R^p: \sum_{j=1}^pv_j^\sharp\sqrt{\log(2p/j)} \leq r ,\norma{\Sigma^{1/2}\vv}_2=1\right\}
        \ \mbox{ and } \
        r = \frac{1+c_0}{\kappa} \sqrt{s\log(2ep/s)}.
    \end{equation}
    It follows from \cite{MR3354613, bednorz, shahar_proc} (cf., for instance,
    Theorem~1.12 in \cite{shahar_proc}) that
    for all $u>0$,
    with probability at least $1-2\exp(-C_2\min(u^2, u\sqrt n) )$,
\begin{equation}
    \label{eq:proc_quad}
    \sup_{\vv\in T}
    \left|\frac{1}{n}\sum_{i=1}^n
        (\vv^T \vx_i)^2-\E\left[(\vv^T \vx)^2\right]
    \right|
    \leq C_1\left(
        \frac{ L \gamma}{\sqrt{n}}
        + \frac{\gamma^2}{n}
        + \frac{u L^2}{\sqrt n}
    \right),
\end{equation}
where
$\gamma = \E [\sup_{\vv\in T} G_\vv ]$
and $(G_\vv)_{\vv\in T}$ is a centered Gaussian process indexed by $T$ with covariance structure given by
    $\E[G_\vv G_\vu] = \vv^T \Sigma \vu$ for all $\vu,\vv\in T$. For instance, one can take $G_\vv = \vv^\top \Sigma^{1/2} \vz$ for $\vv\in T$,  where $\vz\sim \cN(\vzero, I_{p\times p})$. Then, $\gamma = \E\sup_{\vv\in T} \vv^\top \Sigma^{1/2} \vz$.

By  \eqref{eq:proc_quad}, if we take $u = \sqrt n / (4C_1 L^2)$
and if the number of observations $n$ satisfies
$n\geq 64{(C_1\vee C_1^2)} L^2 \gamma^2$,
then  with probability at least $1-2\exp(-C_3 n / L^4)$,
\begin{equation*}
    \frac{1}{2}\leq \frac{1}{n}\sum_{i=1}^n (\vv^T \vx_i)^2\leq \frac{3}{2},
    \qquad
    \forall \vv\in T.
\end{equation*}
Next, we  evaluate the Gaussian mean width $\gamma = \E\sup_{\vv\in T} \vv^\top \Sigma^{1/2} \vz$.

\begin{lemma}\label{lem:gauss_mean_width}
    Let $T$ be as in \eqref{T}, with arbitrary $r>0$.
    Let $\vz\sim\mathcal N(\vzero, I_{p\times p})$
    and let $\Sigma\in \R^{p\times p}$ be a positive semi-definite matrix
    with $\max_{j=1,\dots,p}\Sigma_{jj} \le 1$. Then,
    \begin{equation*}
    \E\sup_{\vv\in T} \vv^\top \Sigma^{1/2} \vz \leq 4 r + \sqrt{\pi/2}.
    \end{equation*}
\end{lemma}
\begin{proof}
   In this proof, we set $g_j = \vz^T \Sigma^{1/2} \ve_j$, $j=1,\dots,p$.
    As $\Sigma_{jj}\le 1$, the variance of $g_j$ is at most 1.
    By Proposition \ref{prop:event-slope}, the event
    \begin{equation*}
        \Omega_0 \triangleq
        \left\{
                \max_{j=1,\dots,p} \frac{g_j^\sharp}{\sqrt{\log(2p/j)}} \le  4
        \right\}
    \end{equation*}
    has  probability at least $1/2$. On the event $\Omega_0$,
    for all $\vv\in T$ we have
    \begin{equation*}
        \vv^T\Sigma^{1/2}\vz
        =
        \sum_{j=1}^p g_j v_j
        \le
        \sum_{j=1}^p g_j^\sharp v_j^\sharp
        \le
        4\sum_{j=1}^p v_j^\sharp \sqrt{\log(2p/j)}   \le 4r.
    \end{equation*}
    Set $f (\vz) \triangleq \sup_{\vv\in T} (\vv^T\Sigma^{1/2}\vz)$.
    We have $f(\vz) \le 4r$ with probability at least $1/2$
    and thus $\Med[f(\vz) ] \le 4r$.
    Furthermore, since the constraint $\norma{\Sigma^{1/2}\vv}_2 = 1 $ is satisfied for any $\vv\in T$,
    the function $f(\cdot)$ is 1-Lipschitz. Therefore,  {by Lemma \ref{lemma:median-mean}},
    $|\Med[f(\vz)] - \E[f(\vz)]| \le \sqrt{\pi/2}$.
    \end{proof}

It follows from Lemma~\ref{lem:gauss_mean_width}  and \eqref{eq:proc_quad} that
if $n\geq C_4 L^2 r^2$ then
with probability at least $1-2\exp(-C_5 n/L^4)$,
we have $(1/2)\leq \norm{\design \vv}_n^2\leq 3/2$
for any $\vv\in T$.
By rescaling, for any
$\vv\in \cC_{WRE}(s,c_0)$ we obtain
$
\norm{\design \vv}_n^2\geq (1/2)\norma{\Sigma^{1/2}\vv}_2^2\geq
(\kappa^2/2)\norma{\vv}_2^2
$.
This proves that the second inequality in \eqref{eq:main_sub_gauss_mat} is satisfied if $n\geq C_4 L^2 r^2$.

We now prove the first inequality in \eqref{eq:main_sub_gauss_mat}.
Let $j\in\{1,\ldots,p\}$ and note that $\design \ve_j$ is a vector in $\R^n$
with i.i.d. subgaussian coordinates since
for all $i=1,\dots,n$ and all $t\ge 0$,
$\E \exp  (t\vx_i^T\ve_j) \le \exp(  t^2 L^2 \Sigma_{jj} /2 )$.
Hence, $\norm{\design \ve_j}_n^2 - \Sigma_{jj}$ is a sum of  independent
zero-mean sub-exponential variables.
It follows from Bernstein's inequality that for $u = 1/L^2$,
with probability at least $1-\exp(-C_6n u^2)$,
\begin{equation*}
    \norm{\design \ve_j}_n^2\leq \Sigma_{jj} +  u L^2 \Sigma_{jj}\le 1
\end{equation*}when $\Sigma_{jj}\leq 1/2$.
By the union bound, the condition $\max_{j=1,\dots,p}\norm{\design\ve_j}_n\le 1$ holds
with probability at least $1 - p e^{-C_6n/L^4}\geq 1 - e^{-C_6n/(2L^4)}$ if $n\ge (2L^4/C_6)\log p$.

In conclusion, both inequalities in \eqref{eq:main_sub_gauss_mat} are satisfied if
\begin{equation*}
    n\geq \frac{C_9 (1+c_0)^2L^2}{\kappa^2} s \log\Big(\frac{2ep}{s}\Big) + \frac{2L^4\log p}{C_6}.
\end{equation*}
Since $\kappa^2\le 1$,  $L\ge1$, and $s \log(2ep/s)\ge \log p$
the inequality in the last display is satisfied if \eqref{eq:assum-n-p-s-slope-random-design} holds for some large enough absolute constant $C>0$.
\end{proof}


\section{Subgaussian noise}
\label{s:proof-main-event-subgaussian}

To prove Proposition \ref{prop:sub_gauss_conc},
we need the following lemma.
\begin{lemma}
    \label{lemma:subgaussian-K=4}
    Let $\sigma>0$, $z\sim\mathcal N(0,1)$, and let $\xi_i$ be a random variable
    satisfying \eqref{eq:assum-subgaussian}.
    Then
    \begin{equation*}
        \mathbb P(|\xi_i|>t) \le 4 \; \mathbb P(\sigma |z| > t),
        \qquad
        \forall \ t\ge 0.
    \end{equation*}
\end{lemma}
\begin{proof}
    By homogeneity, it is enough to consider $\sigma = 1$.
    From a standard lower bound on the Gaussian tail probability, cf.
    \cite[Formula 7.1.13]{AbrStegun}, we get
    \begin{equation*}
        \mathbb P(|z|>t) \ge \frac {4\exp(-t^2/2)}{\sqrt{2\pi} (t + \sqrt{4+t^2})}\ge
        \frac{e^{1-t^2}}{4}, \quad \ \forall t\ge 0,
    \end{equation*}
    while if $\xi_i$ satisfies \eqref{eq:assum-subgaussian} with $\sigma = 1$, a
    Chernoff bound yields that $\mathbb P(|\xi_i|>t) \le \exp(1-t^2)$.
\end{proof}

\begin{proof}[Proof of Proposition \ref{prop:sub_gauss_conc}]
    Let $\eta>0$. Let
    $(\eps_1,\dots,\eps_n)$ be a vector of i.i.d. Rademacher variables independent of $\vxi$. The  symmetrization inequality, cf., e.g. \cite[Theorem 2.1]{koltch}, yields 
    \begin{equation*}
        \E \ {\exp} \left(\eta \sup_{\vu\in U} \sum_{i=1}^n \xi_i u_i\right)
        \leq
        \E \ {\exp} \left(\eta \sup_{\vu\in U} \sum_{i=1}^n 2 \eps_i \xi_i u_i\right)
        .
    \end{equation*}
    By Lemma \ref{lemma:subgaussian-K=4},
    we have
    $
        \mathbb P(|{\eps_i}\xi_i|>t) \le K \; \mathbb P(\sigma |z_i| > t)
    $
    for $K=4$
    and for $i=1,\dots,n$.
    It follows from the contraction inequality
    as stated in \cite[Lemma~4.6]{LT:91} that
    \begin{equation*}
        \E \exp\left(\eta \sup_{\vu\in U} \sum_{i=1}^n 2 \eps_i\xi_i u_i\right)
        \leq
        \E \exp\left(2 \eta K \sigma \sup_{\vu\in U} \vz^T\vu \right)
        .
    \end{equation*}
    Since $U$ is a subset of the unit sphere, the function
    $f:\vz\rightarrow \sup_{\vu\in U} \vz^T\vu$
    is $1$-Lipschitz. Thus,
    by \cite[Theorem 5.5]{boucheron2013concentration},
    the right hand side of the previous display
    is bounded from above by
    \begin{equation*}
        \exp\left( 2\eta K \sigma \E \left[ \sup_{\vu\in U}\vz^T\vu \right]  + 2\eta^2K^2 \sigma^2 \right).
    \end{equation*}
    Furthermore, {by Lemma \ref{lemma:median-mean},}   $|\Med[f(\vz)] - \E f(\vz)]| \le \sqrt{\pi/2}$.
    A Chernoff argument completes the proof.
\end{proof}
\vspace{-6mm}
{
\begin{proof}[Proof of Theorem \ref{thm:main-event-subgaussian}]
    Let $N(\cdot)$ be defined in \eqref{eq:nu} and let $\vz$ be a standard normal $\mathcal N(\vzero,I_{n\times n})$ random vector. {It follows from \eqref{eq:nuu} and Proposition \ref{prop:event-slope} that
    \begin{equation}\label{eq:last}
    \Med\left[
            \sup_{\vu\in \R^p: N(\vu)\le1}
            \frac{1}{n}(\sigma\vz)^T\design\vu
        \right] \le 4.
    \end{equation}
    }
    Let $T\subset\R^p$
    be defined in \eqref{eq:def-T-f} with $L = \sqrt n /(\sigma(\sqrt{\pi/2} + \sqrt{2\log(1/\delta_0)}))$.
        Using Proposition \ref{prop:sub_gauss_conc} {and then \eqref{eq:last}} we obtain that, with probability at least $1-\delta_0$,
    \begin{align*}
        \sup_{\vu\in T}
        \frac{1}{n}\vxi^T\design\vu
        &\le
        8 \sigma \Med\left[
            \sup_{\vu\in T}
            \frac{1}{n}\vz^T\design\vu
        \right]
        + \frac{8 \sigma L}{\sqrt n} \left( \sqrt{\pi/2} + \sqrt{2\log(1/\delta_0)}\right)
        \\
        &\le
        32
        + \frac{8 \sigma L}{\sqrt n} ( \sqrt{\pi/2} + \sqrt{2\log(1/\delta_0)})
        = 40.
    \end{align*}
\end{proof}
}

\section{Lasso with universal tuning parameter}
\label{appendix-previous-improved}

\begin{proof}[Proof of \Cref{prop:previous-improved}]
    Let $\vbeta\in\R^p$ be a minimizer of the right hand side of \eqref{eq:oi-previous-improved}
    and let $T$ be the support of  $\vbeta$, so that $|T|\le s$.
    Using inequality
    \eqref{eq:almost-sure}  with $h(\cdot) = 2 \lambda \norma{\cdot}_1$ we get that, almost surely,
    \begin{equation}
        \norm{\design\hbeta - \vf}^2_n
        -
        \norm{\design\vbeta - \vf}^2_n
        \le
        (2/n) \vxi^T\design(\hbeta - \vbeta)
        + 2 \lambda\norma{\vbeta}_1
        - 2 \lambda\norma{\hbeta}_1
        - \norm{\design(\hbeta-\vbeta)}_n^2.
        \label{eq:extra-appendix-start}
    \end{equation}
    Let $\vu= \hbeta - \vbeta$ and define the function $f$ as follows:
    \begin{equation*}
        f(\vx) = \sup_{\vv\in\R^p: \norm{\design\vv}_n = 1}
        \left(
            (1/\sqrt n) \vx^T\design\vv
            + \sqrt n\lambda
            (
                \norma{\vv_T}_1 - \norma{\vv_{T^c}}_1
            )
        \right),
        \qquad
        \vx\in\R^n.
    \end{equation*}
    By simple algebra,
    $2 \lambda\norma{\vbeta}_1
    - 2 \lambda\norma{\hbeta}_1 \le 2\lambda( \norma{\vu_T}_1 - \norma{\vu_{T^c}}_1 ) = \norm{\design\vu}_n 2\lambda( \norma{\vw_T}_1 - \norma{\vw_{T^c}}_1 )$ where $\vw=(1/\norm{\design\vu}_n) \vu$.
    Hence the right hand side of \eqref{eq:extra-appendix-start}
    is bounded from above by
    \begin{equation*}
        2 \norm{\design\vu}_n f(\vxi) /\sqrt n - \norm{\design\vu}_n^2 \le f^2(\vxi) / n.
    \end{equation*}
    Since the function $f$ is 1-Lipschitz, by the Gaussian concentration bound \cite[inequality (1.4)]{LT:91} we have, for all $\delta\in (0,1)$,
    $$\mathbb P\left(f(\vxi)\le \Med[f(\vxi)] + \sigma \sqrt{2\log(1/\delta)}\right) \ge 1-\delta.$$
    To complete the proof, it remains to show that
    \begin{equation}
    \Med[f(\vxi)]  \le
         \sigma \left(\frac{1+\eps}{\kappa(s,c_0)} \sqrt{2s\log p} + \sqrt s + 2.8\right).
        \label{median-f}
    \end{equation}
    Let  $\Pi_T\in\R^{n\times n}$ be the orthogonal projection onto the linear span of $\{\vx_j, j\in T\}$, where $\vx_j=\design e_j$.
    Then almost surely,
    \begin{align*}
        f(\vxi) &=
       \sup_{\vv\in\R^p: \norm{\design\vv}_n = 1} \left[
            (1/\sqrt n) \vxi^T\Pi_T \design\vv
            +
            (1/\sqrt n) \vxi^T(I_{n\times n}-\Pi_T)\design\vv_{T^c}
            + \sqrt n\lambda
            (
                \norma{\vv_T}_1 - \norma{\vv_{T^c}}_1
            )
            \right]
            \\
            &\le
            \norma{\Pi_T \vxi}_2
                   +
     \sup_{\vv\in\R^p: \norm{\design\vv}_n = 1}     \left[
            (1/\sqrt n) \vxi^T(I_{n\times n}-\Pi_T)\design\vv_{T^c}
            + \sqrt n\lambda
            (
                \norma{\vv_T}_1 - \norma{\vv_{T^c}}_1
            )\right].
    \end{align*}
    The random variable $\norma{\Pi_T \vxi}_2^2/\sigma^2$ has a chi-square distribution with at most $s$ degrees of freedom.
    Standard bounds on the tails of the chi-square distribution yield that
    the event $\Omega_1 =
    \{ \norma{\Pi_T \vxi}_2 \le \sigma(\sqrt s +  \sqrt{2\log(50)}) \}$ satisfies $\mathbb P(\Omega_1) \ge 1 - 1/50 = 0.98$.
    Let $Z = \max_{j=1,...,p} |\vxi^T(I_{n\times n}-\Pi_T)\vx_j |$ and
    define $\Omega_2 = \{ Z \le \sigma\sqrt{2\log p} \}$.
    The random variable $Z$ is the maximum of $2p$ centered Gaussian random variables with variance at most $\sigma^2$ and thus
    $\mathbb P(Z> \sigma x) \le 2p e^{-x^2/2} /(x\sqrt{2\pi})$ for all $x>0$. The choice $x=\sqrt{2\log p}$ yields that
    $\mathbb P(\Omega_2)\ge 1- 1/\sqrt{\pi \log p}\ge0.5208$ for all $p\ge 4$. Direct calculation shows that the same bound on $\mathbb P(\Omega_2)$ is true for $p\in \{2,3\}$. Combining these remarks we find that, for all $p\ge 2$,
    \begin{equation*}
        \mathbb P(\Omega_1\cap\Omega_2)
        \ge
        1- \mathbb P(\Omega_1^c) - \mathbb P(\Omega_2^c)
        \ge
        1- 0.02 - 0.4792
        > 1/2.
    \end{equation*}
    Thus, by definition of the median, an upper bound on $\Med[f(\vxi)]$ is given by an upper bound on $f(\vxi)$ on the event $\Omega_1\cap\Omega_2$:
    \begin{align}\label{I.3}
        \Med[f(\vxi)] &\le \sigma(\sqrt s +  2.8) +
       \sup_{\vv\in\R^p: \norm{\design\vv}_n = 1} \left[
            \sigma\sqrt{2\log p} \ | \vv_{T^c}|_1
            + \sqrt n\lambda
            (
                \norma{\vv_T}_1 - \norma{\vv_{T^c}}_1
            )
            \right]
                \end{align}
where we have used that $\sqrt{2\log(50)}\le 2.8$. Recall that $\sqrt n\lambda= (1+\varepsilon)  \sigma\sqrt{2\log p} $. Thus, if $\norma{\vv_{T^c}}_1 > (1+1/\varepsilon) \norma{\vv_T}_1$, the supremum in \eqref{I.3} is negative and \eqref{median-f} follows.
    Finally, if $\norma{\vv_{T^c}}_1 \le (1+1/\varepsilon) \norma{\vv_T}_1$ then, by the definition of the $RE$ constant $\kappa(s,c_0)$, with $c_0=1+1/\varepsilon$ we obtain
    $ \norma{\vv_T}_1 \le \sqrt{s} \norma{\vv_T}_2\le   \sqrt{s} \norm{\design\vv}_n /\kappa(s,c_0)$.
 Using this remark to bound the supremum in \eqref{I.3} proves \eqref{median-f}.
    \end{proof}

\noindent {\bf Acknowledgement.} This work was supported by GENES and by the French National Research Agency (ANR) under the grants
IPANEMA (ANR-13-BSH1-0004-02) and Labex Ecodec (ANR-11-LABEX-0047). It was also supported by the "Chaire Economie et Gestion des Nouvelles Donn\'ees", under the auspices of Institut Louis Bachelier, Havas-Media and Paris-Dauphine.

\begin{footnotesize}
\bibliographystyle{plain}
\bibliography{biblio}

\begin{thebibliography}{10}

\bibitem{abramovich2010}
F.~Abramovich and V.~Grinshtein.
\newblock M{AP} model selection in gaussian regression.
\newblock {\em Electron. J. Statist.}, 4:932 -- 949, 2010.

\bibitem{AbrStegun}
M.~Abramowitz and I.~A. Stegun.
\newblock {\em Handbook of mathematical functions with formulas, graphs, and
  mathematical tables}.
\newblock National Bureau of Standards Applied Math. Ser., vol.~55. Washington,
  D.C., 1964.

\bibitem{MR3127849}
Q.~Berthet and P.~Rigollet.
\newblock Optimal detection of sparse principal components in high dimension.
\newblock {\em Ann. Statist.}, 41(4):1780--1815, 2013.

\bibitem{MR2533469}
P.~J. Bickel, Y.~Ritov, and A.~B. Tsybakov.
\newblock Simultaneous analysis of {L}asso and {D}antzig selector.
\newblock {\em Ann. Statist.}, 37(4):1705--1732, 2009.

\bibitem{MR3418717}
M.~Bogdan, E.~van~den Berg, C.~Sabatti, W.~Su, and E.~J. Cand{\`e}s.
\newblock S{LOPE}---adaptive variable selection via convex optimization.
\newblock {\em Ann. Appl. Stat.}, 9(3):1103--1140, 2015.

\bibitem{boucheron2013concentration}
S.~Boucheron, G.~Lugosi, and P.~Massart.
\newblock {\em Concentration inequalities: A nonasymptotic theory of
  independence}.
\newblock Oxford University Press, 2013.

\bibitem{candes_davenport2013}
E.~J. Candes and M.A. Davenport.
\newblock How well can we estimate a sparse vector?
\newblock {\em Applied and Computational Harmonic Analysis}, 34:317 -- 323,
  2013.

\bibitem{MR2300700}
E.~J. Candes and T.~Tao.
\newblock Near-optimal signal recovery from random projections: universal
  encoding strategies?
\newblock {\em IEEE Trans. Inform. Theory}, 52(12):5406--5425, 2006.

\bibitem{dalalyan2017prediction}
Arnak~S Dalalyan, Mohamed Hebiri, Johannes Lederer, et~al.
\newblock On the prediction performance of the lasso.
\newblock {\em Bernoulli}, 23(1):552--581, 2017.

\bibitem{MR3354613}
S.~Dirksen.
\newblock Tail bounds via generic chaining.
\newblock {\em Electron. J. Probab.}, 20:no. 53, 29, 2015.

\bibitem{giraud2014introduction}
C.~Giraud.
\newblock {\em Introduction to high-dimensional statistics}, volume 139 of {\em
  Monographs on Statistics and Applied Probability}.
\newblock CRC Press, Boca Raton, FL, 2015.

\bibitem{koltch}
V.~Koltchinskii.
\newblock {\em Oracle Inequalities in Empirical Risk Minimization and Sparse
  Recovery Problems}.
\newblock Ecole d'Ete de Probabilites de Saint-Flour XXXVIII-2008. Springer,
  New York e.a., 2011.

\bibitem{koltchinskii2011nuclear}
V.~Koltchinskii, K.~Lounici, and A.~B. Tsybakov.
\newblock Nuclear-norm penalization and optimal rates for noisy low-rank matrix
  completion.
\newblock {\em Ann. Statist.}, 39(5):2302--2329, 2011.

\bibitem{Shahar-Vladimir}
V.~Koltchinskii and S~Mendelson.
\newblock Bounding the {S}mallest {S}ingular {V}alue of a {R}andom {M}atrix
  {W}ithout {C}oncentration.
\newblock {\em Int. Math. Res. Not. IMRN}, (23):12991--13008, 2015.

\bibitem{LM_compressed}
G.~Lecu{\'e} and S.~Mendelson.
\newblock Sparse recovery under weak moment assumptions.
\newblock Technical report, 2014.
\newblock To appear in Journal of the European Mathematical Society.

\bibitem{LM_reg_sparse}
G.~Lecu{\'e} and S.~Mendelson.
\newblock Regularization and the small-ball method i: Sparse recovery.
\newblock Technical report, CNRS, ENSAE and Technion, I.I.T., 2015.

\bibitem{LT:91}
M.~Ledoux and M.~Talagrand.
\newblock {\em Probability in {B}anach spaces}.
\newblock Springer-Verlag, Berlin, 1991.

\bibitem{lptv2011}
K.~Lounici, M.~Pontil, A.B. Tsybakov, and Sara~A. van~de Geer.
\newblock Oracle inequalities and optimal inference under group sparsity.
\newblock {\em Ann. Statist.}, 39:2164 -- 2204, 2011.

\bibitem{Shahar-COLT}
S.~Mendelson.
\newblock Learning without concentration.
\newblock In {\em Proceedings of the 27th annual conference on Learning Theory
  COLT14}, pages pp 25--39. 2014.

\bibitem{shahar_proc}
S.~Mendelson.
\newblock Upper bounds on product and multiplier empirical processes.
\newblock Technical report, Technion, I.I.T., 2014.
\newblock To appear in Stochastic Processes and their Applications.

\bibitem{MR3367000}
S.~Mendelson.
\newblock Learning without concentration.
\newblock {\em J. ACM}, 62(3):Art. 21, 25, 2015.

\bibitem{raskutti2011}
G.~Raskutti, M.~Wainwright, and B.~Yu.
\newblock Minimax rates of estimation for high-dimensional linear regression
  over $\ell _q$-balls.
\newblock {\em IEEE Trans. Inform. Theory}, 57:6976 -- 6994, 2011.

\bibitem{rigollet2011exponential}
P.~Rigollet and A.~B. Tsybakov.
\newblock Exponential screening and optimal rates of sparse estimation.
\newblock {\em Ann. Statist.}, 39(2):731--771, 2011.

\bibitem{MR3485953}
W.~Su and E.~J. Cand{\`e}s.
\newblock S{LOPE} is adaptive to unknown sparsity and asymptotically minimax.
\newblock {\em Ann. Statist.}, 44(3):1038--1068, 2016.

\bibitem{tsy09}
A.B. Tsybakov.
\newblock {\em Introduction to nonparametric estimation}.
\newblock Springer, New York e.a., 2009.

\bibitem{verzelen2012}
N.~Verzelen.
\newblock Minimax risks for sparse regressions: Ultra-high dimensional
  phenomenons.
\newblock {\em Electron. J. Statist.}, 6:38 -- 90, 2012.

\bibitem{bednorz}
B.~Witold.
\newblock Concentration via chaining method and its applications.
\newblock Technical report, University of Warsaw, 2013.
\newblock ArXiv:1405.0676.

\bibitem{ye_zhang2010}
F.~Ye and C.-H. Zhang.
\newblock Rate minimaxity of the {L}asso and {D}antzig selector for the
  $\ell_q$ loss in $\ell_r$ balls.
\newblock {\em Journal of Machine Learning Research}, 11:3519 -- 3540, 2010.

\end{thebibliography}
\end{footnotesize}

\end{document}